%% file: paper2_biharmonic.tex
\documentclass{nlaauth}

\usepackage[pdftex,plainpages=false,pdfpagelabels,hypertexnames=false,colorlinks=true,pdfstartview=FitV,linkcolor=red,citecolor=blue,urlcolor=blue]{hyperref}

\usepackage{color,graphicx}
\graphicspath{
{./figs/}
{../figs/}
}

\newcommand{\R}{\mathbb{R}}
\newcommand{\cN}{\mathcal{N}}

\def\Lamalp{\tilde\Lambda(m)}
\def\delt{\tilde{\delta}}
\def\deltt{\tilde{\delta}^{t}}

\newcommand{\cO}{\mathcal{O}}

\numberwithin{equation}{section}
\newtheorem{theorem}{Theorem}[section]
\newtheorem{lemma}{Lemma}[section]
\newtheorem{remark}{Remark}[section]

\newenvironment{proof}[1][Proof]{\begin{trivlist}
\item[\hskip \labelsep {\bfseries #1}]}{\end{trivlist}}

\newcommand{\qed}{\nobreak \ifvmode \relax \else
      \ifdim\lastskip<1.5em \hskip-\lastskip
     \hskip1.5em plus0em minus0.5em \fi \nobreak
      \vrule height0.75em width0.5em depth0.25em\fi}

\begin{document}
\NLA{1}{19}{00}{28}{09}

\runningheads{Burak Aksoylu and Zuhal Yeter}
{Robust preconditioners for high-contrast biharmonic equation}

\title{Robust multigrid preconditioners for the high-contrast biharmonic
plate equation\footnotemark[2]}

\author{Burak Aksoylu\affil{1}\corrauth, Zuhal Yeter\affil{1}}
\address{\affilnum{1}\ Department of Mathematics \& 
Center for Computation and Technology,
Louisiana State University
}

\corraddr{Burak Aksoylu: Center for Computation and Technology,
Louisiana State University, 216 Johnston Hall, Baton Rouge LA, 70803 USA}
\footnotetext[2]{Email: {\tt burak@cct.lsu.edu}
}

\received{2 October 2009}
\noaccepted{}


\begin{abstract}
We study the high-contrast biharmonic plate equation with HCT and
Morley discretizations. We construct a preconditioner that is robust
with respect to contrast size and mesh size simultaneously based on
the preconditioner proposed by Aksoylu et al. (2008,
Comput. Vis. Sci. 11, pp. 319--331).  By extending the devised
singular perturbation analysis from linear finite element
discretization to the above discretizations, we prove and numerically
demonstrate the robustness of the preconditioner.  Therefore, we
accomplish a desirable preconditioning design goal by using the same
family of preconditioners to solve elliptic family of PDEs with
varying discretizations. We also present a strategy on how to
generalize the proposed preconditioner to cover high-contrast elliptic
PDEs of order $2k,~k>2$.  Moreover, we prove a fundamental qualitative
property of solution of the high-contrast biharmonic plate
equation. Namely, the solution over the highly-bending island becomes
a linear polynomial asymptotically.  The effectiveness of our
preconditioner is largely due to the integration of this qualitative
understanding of the underlying PDE into its construction.
\end{abstract}


\section{INTRODUCTION}
We study the construction of robust preconditioners for the
high-contrast biharmonic plate equation (also referred as the
biharmonic equation).  The aim is to achieve robustness with respect
to the contrast size and the mesh size simultaneously, which we call
as $m$- and $h$-robustness, respectively.  In the case of a
high-contrast diffusion equation, we studied the family of
preconditioners $B_{AGKS}$ by proving and numerically demonstrating
that the same family used for finite element discretization
\cite{AGKS:2007} can also be used for conservative finite volume
discretizations with minimal modification \cite{AkYe2009}. In this
article, we extend the applicability of $B_{AGKS}$ even further and
show that the very same preconditioner can be used for a wider family
of elliptic PDEs. The broadness of the applicability of $B_{AGKS}$ has
been achieved by singular perturbation analysis (SPA) as it provides
valuable insight into qualitative nature of the underlying PDE and its
discretizations.  In order to study the robustness of $B_{AGKS}$, we
use an SPA that is similar to the one devised on the matrix entries by
Aksoylu et al.~\cite{AGKS:2007}.  SPA turned out to be an effective
tool in analyzing certain behaviors of the discretization matrix
$K(m)$ such as the asymptotic rank, decoupling, low-rank perturbations
(LRP) of the resulting submatrices.  LRPs are exploited to accomplish
dramatic computational savings and this is the main numerical linear
algebra implication.

The devised SPA is utilized to explain the properties of the
submatrices related to $K(m)$.  In particular, SPA of highly-bending
block $K_{HH}(m)$, as modulus of bending $m \to \infty$, has important
implications for the behaviour of the Schur complement $S(m)$ of
$K_{HH}(m)$ in $K(m)$. Namely,
\begin{equation} \label{limitingS}
S(m) := K_{LL} - K_{LH} K_{HH}^{-1}(m) K_{HL} = S_\infty + \mathcal{O}(m^{-1}) \ ,
\end{equation}
where $S_\infty$ is a LRP of $K_{LL}$. The rank of the perturbation
depends on the number of disconnected components comprising the
highly-bending region.  This special limiting form of $S(m)$ allows
us to build a robust approximation of $S(m)^{-1}$ by merely using
solvers for $K_{LL}$ by the help of the
Sherman-Morrison-Woodbury formula.


Preconditioning for the biharmonic equation was extensively studied in
the domain decomposition setting
\cite{Mihajlovic.M;Silvester.D2004,Zhang.X1994} and multigrid, BPX,
and hierarchical basis settings
\cite{Braess.D;Peisker.P1987,Hanisch.M1993,Mihajlovic.M;Silvester.D2004a,Maes.J;Bultheel.A2006,Oswald.P1992,Oswald.P1995}.
Other solution strategies were also developed such as fast Poisson
solvers \cite{mayo1984,mayoGreenbaum1992} and iterative methods
\cite{dang2006}.  However, there is only limited preconditioning
literature available for discontinuous coefficients.  Marcinkowski
\cite{marcinkowski2007} studied domain decomposition preconditioners
for the mortar type discretization of the biharmonic equation with
large jumps in the coefficients. 

The high-contrast in material properties is ubiquitous in composite
materials.  Hence, the modeling of composite materials is an immediate
application of the biharmonic plate equation with high-contrast coefficients.
Since the usage of composite materials is steadily increasing, the
simulation and modeling of composite has become essential.  We witness
that the utilization of composites has become an industry standard.
For instance, light weight composite materials are now being used in
modern aircrafts by Airbus and Boeing.  There is imminent need for
robust preconditioning technology in the computational material
science community as the modeling and simulation capability of
composites evolve.


In \cite{wang2005_masterThesis}, the
Euler-Bernoulli equation with discontinuous coefficients was studied
for the kinematics of composite beams.  In the beam setting, the
physical meaning of the PDE coefficient corresponds to the product of
Young's modulus and moment of inertia
\cite{pozrikidis2005_book}[p. 103], \cite{wang2005_masterThesis}.  In
the biharmonic plate equation setting, the PDE coefficient represents the plate
modulus of bending \cite{pozrikidis2005_book}[p. 406].
Nonhomogeneous elastic plates has been considered in 
\cite{manolisRangelovShaw2003} with varying modulus of elasticity.

Our model problem is limited to the biharmonic equation which captures
only the \emph{isotropic} materials. The extension of our analysis to
a more generalized 4-th order PDE is widely open.  Such PDEs have an
important role in structural mechanics as they are used in modeling
\emph{anisotropic} materials.  Plane deformations of anisotropic
materials were studied in \cite{millerHorgan1995}, but extension to
simultaneously heterogeneous and anisotropic case needs to be further
explored.  Grossi \cite{grossi2001} has studied the existence of the
weak solutions of anisotropic plates. The coercivity of the bilinear
forms has also been established which may lay the foundations for our
future work related to LRPs.

The remainder of the article is structured as follows. In
\S\ref{sec:underlyingPDE}, we present the underlying high-contrast
biharmonic plate equation and the associated bilinear forms. Subsequently,
the effects of high-contrast on the spectrum of stiffness matrix and
its subblocks are also discussed. Since the proposed preconditioner is
based on LRP, in \S\ref{sec:LRP}, we study the LRP of the limiting
Schur complement as in \eqref{limitingS}. In \S\ref{sec:SPA}, we
present the aforementioned SPA and reveal the asymptotic qualitative
nature of the solution. In particular, the solution over the
highly-bending region converges to a linear polynomial as $m
\rightarrow \infty$.  In \S\ref{sec:AGKS}, we introduce the proposed
preconditioner and prove its effectiveness by establishing a spectral bound for
the preconditioned system. In \S\ref{sec:generalPDE}, a strategy is presented
on how to generalize the proposed preconditioner to cover
high-contrast elliptic PDEs of order $2k,~k>2$. In
\S\ref{sec:numerics}, the $m$- and $h$-robustness of the
preconditioner are demonstrated by numerical experiments.

\section{THE UNDERLYING PDE AND THE LINEAR SYSTEM} \label{sec:underlyingPDE}

\begin{figure}[htbp]
\centering{
\includegraphics[width=1in]{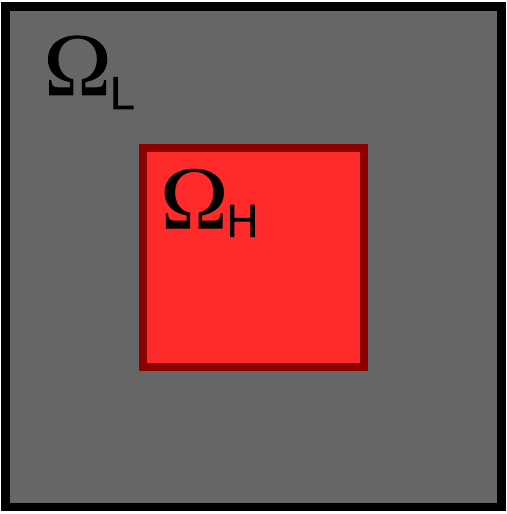}}
\caption{$\Omega = \overline{\Omega}_H \cup \Omega_L$ where $\Omega_H$
  and $\Omega_L$ are highly- and lowly-bending regions,
  respectively.\label{fig:domain1}}
\end{figure}

We study the following high-contrast biharmonic equation for the clamped
plate problem:
\begin{equation} \label{mainProblem}
\begin{array}{rcll}
\nabla^2 \, (\alpha \, \nabla^2 u) & = & 
f \quad & \text{in} \quad \Omega \subset \R^2, \\
u = \partial_n u & = & 0 \quad & \text{on} \quad \partial \Omega.
\end{array}
\end{equation}
We restrict the plate bending process to a \emph{binary regime} (see
Figure \ref{fig:domain1}) 
in which the coefficient $\alpha$ is a
piecewise constant function with the following values:
\begin{equation*}
\alpha(x) =
\begin{cases}
m \gg 1, & x \in \Omega_H, \\
1, & x \in \Omega_L.
\end{cases}
\end{equation*}
It is quite common to idealize the discontinuous PDE coefficient
$\alpha$ by a piecewise constant
function~\cite{BaKn1990,KnWi:2003}.  In the case of high-contrast
diffusion equation, Aksoylu and Beyer~\cite{AkBe2008} showed that the
idealization of diffusivity by piecewise constant coefficients is
meaningful by showing a continuous dependence of the solutions on the
diffusivity; also see~\cite{AkBe2009}.  A similar justification can be
extended to the high-contrast biharmonic plate equation.

\subsection{Bilinear forms for the biharmonic equation}

In the theory of elasticity, potential energy is defined
by using \emph{rotationally invariant} functions.
For plates, the potential energy is given by~\cite[p. 30]{ciarlet2002_book}:
\begin{equation} \label{plateEnergyHessian}
J(v) := \frac{1}{2} \int_{\Omega} \alpha \, 
\left[ \{ \textrm{trace} \, Hess \}^2 + 
2(\sigma-1) \det Hess \right]~dx - \int_{\Omega} fv~dx,
\end{equation}
where $Hess$ is the Hessian, 
\begin{equation*}
Hess = \left[ \begin{matrix} 
\partial_{11}v & \partial_{12}v\\
\partial_{21}v & \partial_{22}v
\end{matrix}
\right].
\end{equation*}
The bilinear form corresponding to energy 
minimization in \eqref{plateEnergyHessian} is given by:
\begin{equation} \label{ciarlet1}
a(u,v) :=  \int_{\Omega} \alpha \, \left [\nabla^2 u \, \nabla^2 v + 
(1 - \sigma) \{ 2 \partial_{12}u  \, \partial_{12}v - 
\partial_{11} u  \, \partial_{22} v - 
\partial_{22} u  \, \partial_{11} v \} \right]~dx,
\end{equation}
where $0< \sigma < 1/2$ is the Poisson's ratio.  Note that the
straightforward bilinear form associated to \eqref{mainProblem} is
obtained by using Green's formula:
\begin{equation} \label{straightforwardBilinearForm}
\int_\Omega \nabla^2 \, (\alpha \, \nabla^2 u) \, v~dx = 
\int_\Omega  \alpha \, \nabla^2 u \, \nabla^2 v~dx +
\int_{\partial \Omega} \alpha \, \partial_n \nabla^2 u \, v~d\gamma -
\int_{\partial \Omega} \alpha \, \nabla^2 u \,  \partial_n v~d\gamma.
\end{equation}
We see that both \eqref{ciarlet1} and \eqref{straightforwardBilinearForm}
contain the so-called \emph{canonical} bilinear form, $\tilde{a}(u,v)$,  
associated to the biharmonic equation \eqref{mainProblem}:
\begin{equation} \label{canonical}
\tilde{a}(u,v) :=  
\int_{\Omega} \alpha \, \nabla^2 u \, \nabla^2 v ~dx.
\end{equation}
When $u,v \in H_0^2(\Omega)$, both bilinear forms $a(u,v)$ and 
$\tilde{a}(u,v)$ correspond to the strong 
formulation \eqref{mainProblem} due to second Green's formula
and the zero contribution of the below term:
\begin{equation} \label{zeroContributionTerm}
 \int_{\Omega} (1 - \sigma) \{ 2 \partial_{12}u  \, \partial_{12}v - 
\partial_{11} u  \, \partial_{22} v - 
\partial_{22} u  \, \partial_{11} v \}~dx.
\end{equation}

\subsection{Effects of high-contrast on the spectrum}
\begin{figure}[th!]
\centering{
\includegraphics[width=6in]{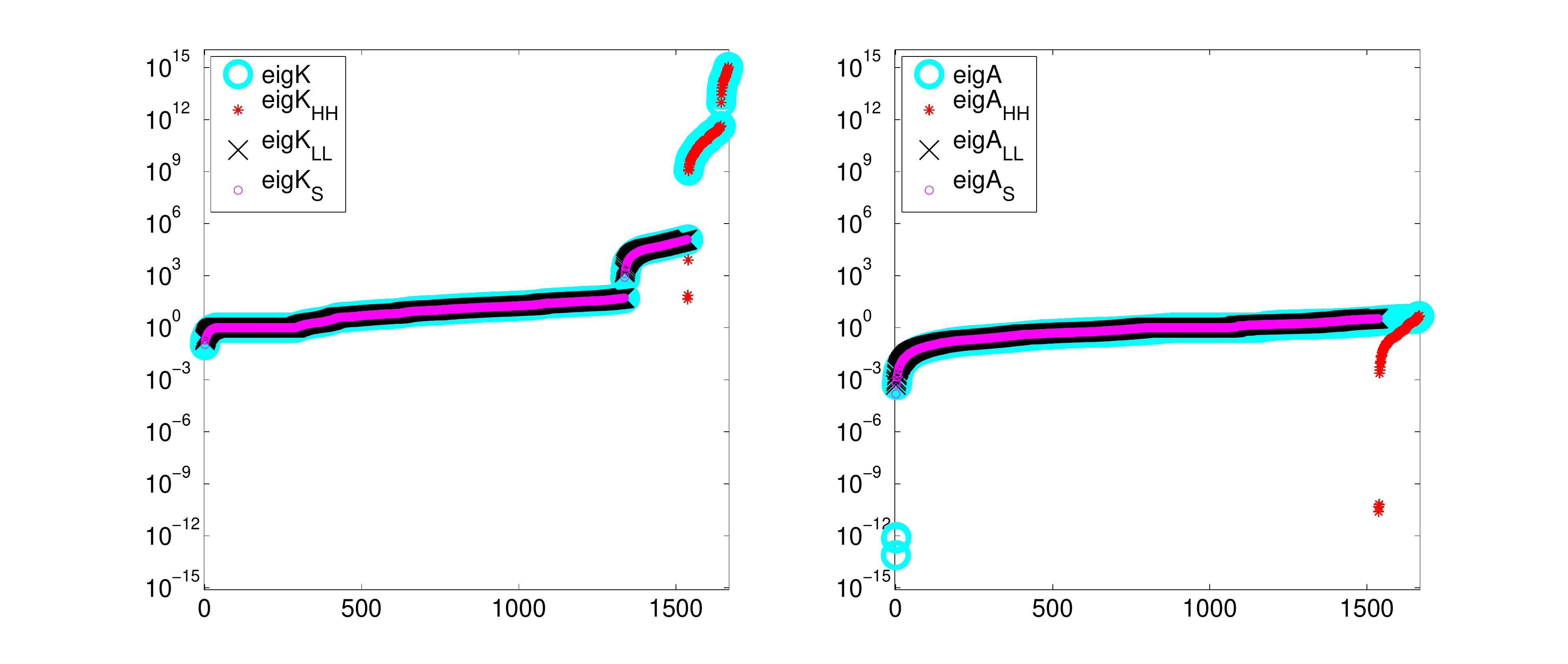}}
\caption{ The HCT discretization of the biharmonic equation with
  $m=10^{10}$.  (Left) The spectrum of the stiffness matrix $K$.
  (Right) Spectrum of the diagonally scaled stiffness matrix. Notice
  the 3 small eigenvalues of order $\mathcal{O}(m^{-1})$ corresponding
  to the kernel of the Neumann matrix, $\textrm{span} \{
  \underline{1}_H, \underline{x}_H, \underline{y}_H \}.$ The plot of
  the two of smallest eigenvalues overlap because they are roughly of
  the same magnitude. \label{fig:spectra}}
\end{figure}
Roughness of PDE coefficients causes loss of robustness of
preconditioners.  This is mainly due to clusters of eigenvalues with
varying magnitude.  Although diagonal scaling has no effect on the
asymptotic behaviour of the condition number, it leads to an improved
clustering in the spectrum.  The spectrum of diagonally scaled
stiffness matrix, $A$, is bounded from above and below except three
eigenvalues in the case of a single isolated highly-bending island.
On the other hand, the spectrum of $K$ contains eigenvalues
approaching infinity with cardinality depending on the number of DOF
contained within highly-bending island.  For the case of HCT
discretization with $m=10^{10}$, we depict the spectra of $K$ and $A$
and their subblocks in Figure \ref{fig:spectra}.  Clustering provided
by diagonal scaling can be advantageous for faster convergence of
Krylov subspace solvers especially when deflation methods designed for
small eigenvalues are used; for further discussion see \cite{AkKl:2007}. 

Utilizing the matrix entry based analysis by Graham and
Hagger~\cite{GrHa:99} for linear FE, in \cite{AkYe2009}, the authors
extended the spectral analysis to cell-centered FV discretization and
obtained an identical spectral result for $A$.  Namely, the number of
small eigenvalues of $A$ depends on the number of isolated islands
comprising the highly-bending region.  We observe a similar behaviour
for the biharmonic plate equation where the only difference is that for each
island we observe three small eigenvalues rather than one.  The three
dimensional kernel of the Neumann matrix is responsible for that
difference; see \S\ref{sec:LRP}.  A similar matrix entry based
analysis can be applied to discretizations of the plate equation, but
this analysis is more involved for HCT and Morley discretizations than
that for linear FE. Hence, we exclude it from scope of this article.

\section{DISCRETIZATIONS AND LOW-RANK PERTURBATIONS}
\label{sec:LRP}

We consider an $H^2$-conformal and also an $H^2$-nonconformal
Galerkin finite element discretization; Hsieh-Clough-Tocher (HCT)
\cite{HCTorigPaper} and Morley \cite{morleyOrigPaper} elements, respectively.
Let the linear system arising from the discretization be denoted by:
\begin{equation} \label{mainLinearSys}
K(m)~x = b.
\end{equation}
$\Omega$ is decomposed with respect to magnitude of the coefficient value as
\begin{equation} \label{subregionDecomp}
\Omega = \overline{\Omega}_H \cup \Omega_L,
\end{equation}
where $\Omega_H$ and $\Omega_L$ denote the highly- and lowly-bending
regions, respectively.  DOF that lie on the interface, $\Gamma :=
\overline{\Omega}_H \cap \overline{\Omega}_L$, between the two regions
are included in $\Omega_H$. When $m$-dependence is explicitly stated and
the discretization system \eqref{mainLinearSys} is decomposed with
respect to \eqref{subregionDecomp}, i.e., the magnitude of the
coefficient values, we arrive at the following $2 \times 2$ block
system:
\begin{equation} \label{2x2blockSys}
\left[
\begin{array}{ll}
K_{HH}(m) & K_{HL} \\ K_{LH} & K_{LL}
\end{array}
\right]
\left[ \begin{array}{c} x_H \\ x_L \end{array} \right]
= \left[ \begin{array}{c} b_H \\ b_L \end{array} \right].
\end{equation}
There are important properties associated to the $K_{HH}$ block in
\eqref{2x2blockSys}: It is the only block that has $m$-dependence, and
furthermore, a matrix with low-rank kernel can be extracted from it.
Our preconditioner construction is based on LRPs from this extraction.
Next, we explain how to extract the so-called \emph{Neumann matrix}
and why $a(u,v)$ is the suitable bilinear form for that purpose.

By rewriting \eqref{ciarlet1} as the following
\begin{equation} \label{ciarlet2}
a(u,v) =  \int_{\Omega} \alpha \, \left [ \sigma \, \nabla^2 u \, \nabla^2 v + 
(1 - \sigma) \{ \partial_{11}u  \, \partial_{11}v +
\partial_{22} u  \, \partial_{22} v + 
2 \, \partial_{12} u  \, \partial_{12} v \} \right]~dx,
\end{equation}
we see that
\begin{eqnarray}
a(v,v) & = & \alpha \, \sigma~\|\nabla^2 v\|^2_{L_2(\Omega)} + 
\alpha \, (1-\sigma) |v|^2_{H^2(\Omega)} \nonumber \\
& \geq & \alpha \, (1-\sigma) |v|^2_{H^2(\Omega)} \label{gateway}.
\end{eqnarray}
The inequality \eqref{gateway} has important implications. Namely,
$a(v, v)$ is $V_{\mathcal{P}_1}(\Omega)$-coercive where
$V_{\mathcal{P}_1}(\Omega) \subset H^2(\Omega)$ is a closed subspace
such that $V_{\mathcal{P}_1}(\Omega) \cap \mathcal{P}_1 = \emptyset$
and $\mathcal{P}_1$ denotes the set of polynomials of degree at most
$1$.  Furthermore, \eqref{gateway} immediately implies that 
$a(v,v)$ is $H_0^2(\Omega)$-coercive. 

Let $\mathcal{T}^h$ be the triangulation of $\Omega$. Based on
$\mathcal{T}^h$, we define the associated discrete space
$V_{\mathcal{P}_1}^h(\Omega)$ such that $V_{\mathcal{P}_1}^h \cap
\mathcal{P}_1^h = \emptyset$.  A precise definition of the $K_{HH}$
block in the stiffness matrix in \eqref{mainLinearSys} is given by:
\begin{equation*}
\langle K_{HH} \underline{\phi}_H^h, \underline{\psi}_H^h \rangle :=  
a(\phi_H^h,\psi_H^h),
\end{equation*}
where $\phi_H^h, \psi_H^h \in V^h(\Omega_H) \subset H_0^2(\Omega_H)$
are the basis functions. We define the \emph{Neumann matrix}
$\cN_{HH}$ as follows:
\begin{equation*}
\langle \cN_{HH} \underline{\phi}_H, \underline{\psi}_H \rangle := 
a(\phi_H^h,\psi_H^h),
\end{equation*}
where $\phi_H^h, \psi_H^h \in V_{\mathcal{P}_1}^h (\Omega_H)$. 
Since $a(\cdot, \cdot)$ is
$V_{\mathcal{P}_1}(\Omega)$-coercive, this implies by \eqref{gateway} that
\begin{equation} \label{ker_Nhh}
\ker \cN_{HH} = \mathcal{P}_1^h|_{\overline{\Omega}_H} =~
\textrm{span} \{ \underline{1}_H, \underline{x}_H, \underline{y}_H \}.
\end{equation}
Hence, $K_{HH}(m)$ has the following decomposition:
\begin{equation} \label{K_HH_decomp}
K_{HH}(m) = m \, \cN_{HH} + R,
\end{equation}
where $R$ is the coupling matrix corresponding to DOF on the interface
$\Gamma$.  Now, we are in a position to reveal the resulting main
numerical linear algebra implication.  As $m \rightarrow \infty$, the
limiting Schur complement $S_\infty$ in \eqref{limitingS} becomes a
rank-3 perturbation of $K_{LL}$.  This result relies on the fact that the 
inverse of the limiting $K_{HH}$ is of rank-3; see \eqref{lemma:part_i}.
This is due to the fact that $\cN_{HH}$ has a rank 3 kernel whose (normalized)
discretization is given by:
\begin{equation} \label{defn:e_H}
e_H := [\underline{1}_H, \underline{x}_H, \underline{y}_H].
\end{equation}

\section{MAIN SINGULAR PERTURBATION ANALYSIS RESULTS} \label{sec:SPA}
\begin{lemma} \label{lemma:Main} 
The asymptotic behaviour of the submatrices in \eqref{eq:exact} 
is given by the following:
\begin{eqnarray}
K_{HH}(m)^{-1} & = & e_{H}\eta^{-1} e_{H}^{t} + \mathcal{O}(m^{-1}), 
\label{lemma:part_i}\\
S(m) & = & K_{LL} - 
    (K_{LL} e_{H}) \eta^{-1} (e_{H}^{t} K_{LL}) +
    \mathcal{O}(m^{-1}), \label{lemma:part_ii}\\
K_{LH} K_{HH}(m)^{-1} & = & (K_{L L} e_H) 
\eta^{-1} e_{H}^{t} + \mathcal{O}(m^{-1}), \label{lemma:part_iii}
\end{eqnarray}
where
\begin{equation} \label{defn:eta}
\eta := e_H^t \, K_{HH} \, e_H.
\end{equation} 
\end{lemma}

\begin{proof}
Since $\cN_{HH}$ is symmetric positive semidefinite, using \eqref{ker_Nhh}
we have the following spectral decomposition where
$n_H$ denotes the cardinality of DOF in $\overline{\Omega}_H$:
\begin{equation} \label{spectralDecomp1}
Z^t \cN_{HH} Z =
\text{diag}(\lambda_1, \ldots, \lambda_{n_H-3}, 0, 0, 0),
\end{equation}
where $\{ \lambda_i :\ i = 1, \ldots, n_H \}$ is a non-increasing
sequence of eigenvalues of $\cN_{HH}$ and $Z$ is orthogonal.  Since,
the eigenvectors corresponding to the zero eigenvalues are 
discretization of the polynomials $1, x$, and $y$, we can write
$Z = \left[\tilde{Z} \ | \ e_H \right]$ where $e_H$ is defined
in \eqref{defn:e_H}. Using \eqref{K_HH_decomp}, we have:
\begin{eqnarray}
Z^{t}K_{HH}(m)Z  & = & \left[
\begin{matrix}
  m~\textrm{diag} (\lambda_{1}, \ldots, \lambda_{n_H - 3}) +
  \tilde{Z}^{t}R \tilde{Z} & ~\tilde{Z}^{t}R e_{H}
  \\ e_{H}^{t} R \tilde{Z} & e_{H}^{t} R e_{H} \\
\end{matrix} 
\right] \nonumber \\ 
& =: & \left[
\begin{matrix}
\Lamalp & \delt \\ \deltt & \eta \label{spectralDecomp1b} \\
\end{matrix}
\right].
\end{eqnarray}

To find the limiting form of $K_{HH}(m)^{-1}$ note that
\begin{eqnarray*}
\tilde{\Lambda}(m) & = & 
m~\textrm{diag}(\lambda_{1}, \ldots, \lambda_{n_H - 3}) + 
\tilde{Z}^{t} R \tilde{Z} \\
&=&
m~\textrm{diag}(\lambda_{1}, \ldots, \lambda_{n_H - 3}) \left( \tilde{I} +
m^{-1}~\textrm{diag}(\lambda_{1}^{-1}, \ldots,
\lambda_{n_H - 3}^{-1})\tilde{Z}^{t} R \tilde{Z} \right).
\end{eqnarray*}
Then, 
\begin{equation*}
\|\tilde{\Lambda}(m)^{-1}\|_2 \nonumber  \leq
\frac{m^{-1} \, \max_{i \leq n_H-3} \ \lambda_i^{-1}} 
{1 - m^{-1} \,  \max_{i \leq n_H-3} \ \lambda_i^{-1} \, 
\|\tilde{Z}^t R \tilde{Z}\|_2 },
\end{equation*}
for sufficiently large $m$, we can conclude the following:
\begin{equation} \label{result1_lemma}
   \tilde{\Lambda}(m)^{-1} =  \cO(m^{-1}).
\end{equation}
We proceed with the following inversion:
\begin{equation*} \label{inversion1} \left[ \begin{array}{cc}
      \tilde{\Lambda}(m) & \delt \\
      \delt^t & \eta
\end{array} \right]^{-1} = U(m)~V(m)~U(m)^t,
\end{equation*}
where
\begin{eqnarray*}
  U(m) & := &
  \left[ \begin{array}{cc}
      \tilde{I} & -\tilde{\Lambda}(m)^{-1} \delt \\
      0^t & 1
    \end{array} \right],\\
  V(m) & := &
  \left[ \begin{array}{cc}
      \tilde{\Lambda}(m)^{-1} & 0 \\ 
      0^t & \left(\eta - \delt^t \tilde{\Lambda}(m)^{-1} \delt
      \right)^{-1}
\end{array} \right].
\end{eqnarray*}

Then, \eqref{result1_lemma}  implies that
\begin{eqnarray*} \label{result3_lemma}
  U(m) & = & I + \cO(m^{-1}), \\
  V(m) & = & \left[ \begin{array}{cc} O & 0 \\ 0^t & \eta^{-1}
\end{array} \right]  + \cO(m^{-1}).
\end{eqnarray*}
Combining the above results, we arrive at
\begin{equation*}
  \label{inversion_limit}
  \left[ \begin{array}{cc}
      \tilde{\Lambda}(m) & \delt \\
      \delt^t & \eta
    \end{array} \right]^{-1}
  \ = \
  \left[ \begin{array}{cc}
      O & 0 \\ 0^t & \eta^{-1}
    \end{array} \right]
  \ + \ \mathcal{O}(m^{-1})\ ,
\end{equation*}
and, by \eqref{spectralDecomp1b}, we have
\begin{eqnarray}
  \label{eq:largest}
  K_{HH}(m)^{-1} \ & = & 
  \ Z \left[ \begin{array}{cc}
      O & 0 \\ 0^t & \eta^{-1}
    \end{array} \right] Z^t \ + \ \mathcal{O}(m^{-1}) \ \\
  & =: & \ e_H \eta^{-1} e_H^t  \ + \  \mathcal{O}(m^{-1}) \ ,\nonumber
\end{eqnarray}
which proves \eqref{lemma:part_i} of the Lemma.

Parts \eqref{lemma:part_ii} and \eqref{lemma:part_iii} follow from simple 
substitution and using \eqref{SchurComplement1}. \hfill \qed \\
\end{proof}

\begin{remark}
If we further decompose DOF associated with $\overline{\Omega}_H$
into a set of interior DOF associated with index $I$ and
interface DOF with index $\Gamma$, we obtain the following block
representation of $K_{HH}$:
\begin{equation} \label{A_HH_blockwise}
K_{HH}(m) \ = \ \left[ \begin{array}{cc}
K_{II}(m) & K_{I \Gamma}(m) \\
K_{\Gamma I}(m) & K_{\Gamma \Gamma}(m)
\end{array} \right].
\end{equation}
The entries in the block $K_{\Gamma \Gamma}(m)$ are
assembled from contributions both from finite elements in $\Omega_H$
and $\Omega_L$, i.e.  $K_{\Gamma \Gamma}(m) = A^{(H)}_{\Gamma
  \Gamma}(m) + A^{(L)}_{\Gamma \Gamma}$. 

We further write $e_H$ in block form; 
$e_H = ( e_I^t \ , \ e_{\Gamma}^t)^t$.
Finally we note that the off-diagonal blocks have the decomposition:
\begin{equation}
\label{A_HL_blockwise}
K_{LH} \ = \ \left[ \begin{array}{cc}
0 & K_{L\Gamma}
\end{array} \right] \ = \ K_{HL}^t.
\end{equation}
Therefore, the results of Lemma \ref{lemma:Main} can be rewritten as
the following:
\begin{eqnarray*}
K_{HH}(m)^{-1} & = & e_{H} 
\left( e_\Gamma^t K_{\Gamma \Gamma}^{(L)} e_\Gamma \right)^{-1} e_{H}^t +
    \mathcal{O}(m^{-1}),\\
S(m) & = & K_{LL} - (K_{L\Gamma} e_{\Gamma}) 
\left( e_\Gamma^t K_{\Gamma \Gamma}^{(L)} e_\Gamma \right)^{-1} 
(e_{\Gamma}^{t} K_{\Gamma L} + \mathcal{O}(m^{-1}),\\
K_{LH} K_{HH}(m)^{-1} & = & (K_{L \Gamma} e_{\Gamma})
\left( e_\Gamma^t K_{\Gamma \Gamma}^{(L)} e_\Gamma \right)^{-1} 
e_{H}^{t} + \mathcal{O}(m^{-1}).
\end{eqnarray*}
\end{remark}

\subsection{Qualitative nature of the solution}
\label{sec:qualitativeNature}

We advocate the usage of SPA because it is a very effective tool in
gaining qualitative insight about the asymptotic behavior of the
solution of the underlying PDE.  Through SPA, in
Lemma~\ref{lemma:Main}, we were able to fully reveal the asymptotic
behaviour of the submatrices of $K$ in \eqref{eq:exact}. This
information leads to a characterization of the limit of the underlying
discretized inverse operator.  We now prove that \emph{the solution
  over the highly-bending island converges to a linear polynomial}. In
other words, $x_H^\infty \in \text{span}~e_H$.  This is probably the
most fundamental qualitative feature of the solution of the
high-contrast biharmonic plate equation.

\begin{lemma}
Let $e_H$ as in \eqref{defn:e_H}. Then, 
\begin{equation} \label{x_HConst}
x_H(m) = e_H~c_H  \ + \ \mathcal{O}(m^{-1}),
\end{equation}
where $c_H$ is a $3 \times 1$ vector determined by the solution in the
lowly-bending region.
\end{lemma}

\begin{proof}
We prove the result by providing an explicit quantification of
the limiting process based on Lemma~\ref{lemma:Main}:
\begin{equation*}
\begin{array}{lllll}
x_L(m) & = & S^{-1}(m)~
  \{b_L - K_{LH} \, K_{HH}^{-1}(m) b_H \} &&\\ 
  & = & S_\infty^{-1} \{b_L - 
  K_{LH} \left( e_H \eta^{-1} e_H^t \right) b_H \} + \mathcal{O}(m^{-1})\\
  & =: & x_L^\infty + \mathcal{O}(m^{-1}),\\
  x_H(m) & = & 
  K_{HH}^{-1}(m)~ \{b_H - K_{HL} \,  x_L(m) \} &&\\ 
  & = & e_H \eta^{-1} e_H^t 
  \{ b_H - K_{HL} \, x_L^\infty \} +
  \mathcal{O}(m^{-1})\\
  & =: & e_H~c_H  \ + \ \mathcal{O}(m^{-1}).
\end{array}
\end{equation*}
\hfill \qed
\end{proof}

\section{CONSTRUCTION OF THE PRECONDITIONER} \label{sec:AGKS}

The exact inverse of $K$ can  be written as:
\begin{eqnarray}
K^{-1} & = &
  { \left[ \begin{array}{cc}  I_{HH} & ~~- K_{HH}^{-1} K_{HL}\\ 
        0 & I_{LL}\end{array}\right]}~
  { \left[ \begin{array}{cc}    K_{HH}^{-1}& 0 \\
        0& ~~S^{-1} \end{array}\right]}~
  { \left[ \begin{array}{cc}  I_{HH} & 0 \\    -K_{LH}K_{HH}^{-1}& ~~I_{LL}
      \end{array}\right]}, \label{eq:exact}
\end{eqnarray}
where $I_{HH}$ and $I_{LL}$ denote the identity matrices of the
appropriate dimension and the Schur complement $S$ is explicitly given by:
\begin{equation} \label{SchurComplement1}
S(m) = K_{LL} - K_{LH}K_{HH}^{-1}(m) K_{HL}.
\end{equation} 

Let the limit in \eqref{lemma:part_i} be denoted by 
$K_{HH}^{\infty^{\dagger}}:=e_{H}\eta^{-1} e_{H}^{t}$. Based on
the above perturbation analysis, our proposed preconditioner is
defined as follows:
\begin{equation} \label{MainPrec1}
B_{AGKS}(m):= 
\left[ \begin{array}{cc}  I_{HH} & - K_{HH}^{\infty^\dagger} K_{HL} \\
0 & I_{LL}\end{array}\right]
\left[ \begin{array}{cc}    K_{HH}(m)^{-1}& 0 \\
0&  S_\infty^{-1} \end{array}\right]
\left[ \begin{array}{cc}  I_{HH} & 0 \\    - K_{LH}K_{HH}^{\infty^\dagger} & I_{LL}
\end{array}\right]
\end{equation}

We need the following auxillary result to be used in the proof of
Theorem \ref{thm:prec_robust} which characterizes the spectral
behaviour of the preconditioned system.
\begin{lemma} 
For sufficiently large $m$, we have
\begin{equation}
\label{A_HH0.5}
K_{HH}^{-1/2} = e_H \eta^{-1/2} e_H^t + \cO(m^{-1/2}),
\end{equation}
where $\eta$ is the $3 \times 3$ SPD matrix independent of $m$ defined 
in \eqref{defn:eta}.
\end{lemma} 

\begin{proof}
We start by writing down the spectral decomposition of $K_{HH}(m)$
\begin{equation*}
\label{eq:eigdecomp}
Q(m)^t K_{HH}(m) Q(m) \ = \
\ \mathrm{diag}(\mu_1(m), \ldots , \mu_{n_H-3}(m), \mu_{n_H-2}(m) , 
\mu_{n_H-1}(m), \mu_{n_H}(m)),
\end{equation*}
where $\{\mu_i(m): \ i = 1, \ldots , n_H\}$ denotes a non-increasing
ordering of the eigenvalues of $K_{HH}(m) $. Since $K_{HH}(m)$ is SPD,
we have $\mu_i(m) > 0 $ for all $i\leq n_H$.  We use the main fact
that eigenvalues and eigenvectors of a symmetric matrix are Lipschitz
continuous functions of the matrix
entries~\cite{kato1982_book,watkins2002_book}.

By \eqref{spectralDecomp1} and \eqref{eq:largest} in Lemma \ref{lemma:Main},
we give the following spectral decomposition:
\begin{equation} \label{interStep2}
K_{HH}^{-1}(m) = z_1 \, 0 \, z_1^t + \ldots + z_{n_{H}-3} \, 0 \, z_{n_{H}-3}^t 
+ e_H \, \eta^{-1} \, e_H^t + \cO(m^{-1}).
\end{equation}
Note that $\eta$ in \eqref{spectralDecomp1b} is a $3 \times 3$ symmetric, 
and hence, diagonalizable matrix. 
We proceed towards a fully diagonalized form of the limiting $K_{HH}^{-1}(m)$. 
For that, we use the diagonalization of $\eta^{-1}$:
\begin{equation*} \label{interStep3}
\eta^{-1} = \hat{z}_{H_1} \, \mu_{H_1}^{-1} \, \hat{z}_{H_1}^t +  
\hat{z}_{H_x} \, \mu_{H_x}^{-1} \, \hat{z}_{H_x}^t + 
\hat{z}_{H_y} \, \mu_{H_y}^{-1} \, \hat{z}_{H_y}^t.
\end{equation*}
Therefore, we have the following expression for the last term in 
\eqref{interStep2}:
\begin{equation} \label{lastTermDiag}
e_H \eta^{-1} e_H^t = [z_{H_1} \, z_{H_x} \, z_{H_y}] \,
\text{diag}(\mu_{H_1}^{-1}, \mu_{H_x}^{-1}, \mu_{H_y}^{-1}) \, 
 [z_{H_1} \, z_{H_x} \, z_{H_y}]^t,
\end{equation}
where 
\begin{eqnarray*}
\left[z_{H_1} \, z_{H_x} \, z_{H_y}\right] & := & 
\left[e_{H_1} \, e_{H_x} \, e_{H_y}\right] 
\left[\hat{z}_{H_1} \, \hat{z}_{H_x} \, \hat{z}_{H_y}\right]\\
\left[e_{H_1}, e_{H_x}, e_{H_y}\right] & := & e_H.
\end{eqnarray*}

Now by substituting \eqref{lastTermDiag} in \eqref{interStep2},
we have the following spectral decomposition which corresponds to the fully
diagonalized version:
\begin{eqnarray} 
K_{HH}^{-1}(m) & = & z_1 \, 0 \, z_1^t + \ldots + 
z_{n_{H}-3} \, 0 \, z_{n_{H}-3}^t +
z_{H_1} \, \mu_{H_1} \, z_{H_1}^t  + z_{H_x} \, \mu_{H_x} \, z_{H_x}^t  +
z_{H_y} \, \mu_{H_y} \, z_{H_y}^t  + \cO(m^{-1}) \nonumber \\
& =: & Z_\infty \, 
\text{diag}(0, \ldots, 0, \mu_{H_1}^{-1}, \mu_{H_x}^{-1}, \mu_{H_y}^{-1}) \, Z_\infty^t
+ \cO(m^{-1}) \label{fullDiagVersion}.
\end{eqnarray}
The expression in \eqref{fullDiagVersion} also implies the convergence
of the eigenvectors of $K_{HH}(m)$:
\begin{equation} \label{eigvecConv}
Q(m) = Z_\infty + \cO(m^{-1}).
\end{equation}
Note that $Z_\infty$ differs from $Z$ in \eqref{spectralDecomp1} only
in the last three columns due to diagonalization of $\eta$.

From \eqref{fullDiagVersion}, we obtain a characterization of the
largest three eigenvalues of $K_{HH}(m)^{-1}$:
\begin{subequations} \label{etaSpectralDecomp}
\begin{eqnarray} 
\mu_{n_H - 2}(m)^{-1} & = & \mu_{H_1}^{-1} + \mathcal{O}(m^{-1})\\
\mu_{n_H - 1}(m)^{-1} & = &  \mu_{H_x}^{-1} + \mathcal{O}(m^{-1})\\ 
\mu_{n_H}(m)^{-1} & = &  \mu_{H_y}^{-1} + \mathcal{O}(m^{-1})\ .
\end{eqnarray}
\end{subequations}
Using \eqref{fullDiagVersion} and \eqref{etaSpectralDecomp}, we arrive at the 
following:
\begin{eqnarray}
& & \text{diag} (\mu_{1}(m)^{-1/2}, \ldots, \mu_{n_H - 3}(m)^{-1/2}, 
\mu_{n_H - 2}(m)^{-1/2}, \mu_{n_H - 1}(m)^{-1/2}, \mu_{n_H}(m)^{-1/2}) \nonumber \\
& & =  
\text{diag} (0, \ldots, 0, \mu_{H_1}^{-1/2}, \mu_{H_x}^{-1/2}, \mu_{H_y}^{-1/2})
+ \cO(m^{-1/2}). \label{interStep4}
\end{eqnarray}
By using \eqref{interStep4} and \eqref{eigvecConv}, we arrive at the 
desired result:
\begin{eqnarray*}
K_{HH}(m)^{-1/2}  & = & Q(m) \, 
\text{diag} (\mu_{1}(m)^{-1/2}, \ldots, \mu_{n_H}(m)^{-1/2})
Q(m)^t \\
& = & Z_\infty \,  \text{diag} (0, \ldots, 0, \mu_{H_1}^{-1/2}, \mu_{H_x}^{-1/2}, \mu_{H_y}^{-1/2}) \, Z_\infty^t
+ \cO(m^{-1/2})\\
& = &  [z_{H_1} \, z_{H_x} \, z_{H_y}] \,
\text{diag} ( \mu_{H_1}^{-1/2}, \mu_{H_x}^{-1/2}, \mu_{H_y}^{-1/2}) \,
[z_{H_1} \, z_{H_x} \, z_{H_y}]^t + \cO(m^{-1/2}) \\
& = & e_H \, \eta^{-1/2} \, e_H^t + \cO(m^{-1/2}).
\end{eqnarray*}
\hfill \qed \\
\end{proof}

The following theorem shows that $B_{AGKS}$ is an effective preconditioner
for $m \gg 1$.
\begin{theorem} \label{thm:prec_robust}
For sufficiently large $m$, we have
\[
\sigma(B_{AGKS}(m)~K(m)) \ \subset \ [1-cm^{-1/2},1+cm^{-1/2}]
\]
for some constant $c$ independent of $m$, and therefore
\[
\kappa(B_{AGKS}(m)~K(m)) \ = \ 1 \ + \ \cO(m^{-1/2}).
\]
\end{theorem}

\begin{proof}
Let us factorize the preconditioner as $B_{AGKS}=L^tL$ with
\begin{equation*} 
L:= \left[ \begin{matrix} 
K_{HH}(m)^{-1/2} & 0  \\ 
-S_\infty ^{-1/2} \, P_{LH}^\infty & S_\infty ^{-1/2}
\end{matrix} \right],
\end{equation*}
where the limiting Schur complement $S(m)$ and $K_{LH} K_{HH}^{-1}$ is
denoted by $S_\infty$ and $P_{LH}^\infty$, respectively. 
We can easily show that 
\begin{equation} \label{BK}
\sigma(B_{AGKS} K) = \sigma(LKL^t) = \sigma(I + E).
\end{equation}

Note that 
\begin{equation} \label{interStep1}
P_{LH}^\infty K_{HH} P_{LH}^{\infty^t} - P_{LH}^{\infty} K_{HL} =
 K_{LH}(e_H\eta^{-1}e_H^tK_{HH}e_H\eta^{-1}e_H^t -
 e_H\eta^{-1}e_H^t)K_{HL} = 0.
\end{equation}
We give a step of the operation leading to \eqref{BK}.
Using \eqref{interStep1}, the $(2,2)$-th block entry of the $LKL^t$ reads:
\begin{equation*}
S_\infty ^{-1/2} [P_{LH}^{\infty} K_{HH} P_{LH}^{\infty^t} -
 P_{LH}^\infty K_{HL} - K_{LH}P_{LH}^{\infty^t} + K_{LL}] S_\infty ^{-1/2} = I.
\end{equation*}
The other entries of $LKL^t$ can be computed in a similar way.

Using \eqref{A_HH0.5}, we have
\begin{equation*}
E_{LH} \ = \
S_\infty ^{-1/2} K_{LH} (I_{HH} - e_H \eta^{-1} e_H^t K_{HH})
e_H \eta^{-1/2}
e_H^t + \cO(m^{-1/2}) \ = \ \cO(m^{-1/2}). \label{E_LH}
\end{equation*}
Hence $\rho(E)$, the spectral radius of $E$, is $\cO(m^{-1/2})$,
which together with (\ref{BK}) completes the proof.
\hfill \qed \\
\end{proof}

\section{GENERALIZATION TO ELLIPTIC PDES OF ORDER $2k$} \label{sec:generalPDE}

In essence, the biharmonic plate equation preconditioner is an extension of the
construction for the diffusion equation.  It is possible to generalize
this construction to a family of elliptic PDEs of order $2k, k>2$.
We present how to obtain LRPs from associated bilinear forms.
We choose a different perspective than the one in Section \ref{sec:LRP}.
We start with a canonical bilinear form and show the modification
it needs to go through in order to construct LRPs.

Let the generalized problem be stated as follows:
Find $u \in H_0^k(\Omega)$ such that 
\begin{equation} \label{eq:superPlateEq}
T_k u := (-1)^k \, \nabla^k \left( \alpha_k \, \nabla^k u \right) = 
f \qquad \text{in}~ \Omega.
\end{equation}
The straightforward bilinear form associated to
\eqref{eq:superPlateEq} is obtained by application of Green's
formula $k$ times:
\begin{equation} \label{straightforwardBilinearForm_k}
\int_\Omega \nabla^k \, (\alpha_k \, \nabla^k u) \, v~dx = 
\int_\Omega  \alpha_k \, \nabla^k u \, \nabla^k v~dx + 
\quad \text{boundary terms}.
\end{equation}
Then, we define a bilinear form corresponding to 
\eqref{eq:superPlateEq} which can be seen as a \emph{generalization} 
of the \emph{canonical} bilinear form in \eqref{canonical}:
\begin{equation} \label{canonical_k}
\tilde{a}_k(u,v) :=  
\int_{\Omega} \alpha_k \, \nabla^k u \, \nabla^k v ~dx.
\end{equation}
Without modification, $\tilde{a}_k(\cdot, \cdot)$ cannot lead to LRPs
because $\tilde{a}_k(v,v)$ is not $H_0^k(\Omega)$-coercive.  This is
due to the fact that $\tilde{a}_k(v,v)=0$ for $v \in \mathcal{P}_{k-1}
\cap H_0^k(\Omega)$.  Hence, the stiffness matrix induced by
\eqref{canonical_k} has a large kernel involving elements from
$\mathcal{P}_{k-1}^h \cap V^h$ which indicates that extraction
of a Neumann matrix with a low-dimensional kernel is impossible.
In order to overcome this complication, we utilize a modified bilinear form:
\begin{equation*}
a_k(u,v) = \tilde{a}_k(u,v) + (1 - \sigma_k) \, \hat{a}_k(u,v).
\end{equation*}
The bilinear form should maintain the following essential properties:
\begin{enumerate}
\item $H_0^k(\Omega)$-coercive.
\item $V_{\mathcal{P}_{k-1}(\Omega)}$-coercive.
\item Corresponds to a strong formulation giving $T_k u$ in 
\eqref{eq:superPlateEq} precisely,
\end{enumerate}
where $V_{\mathcal{P}_{k-1}(\Omega)}$ is a closed subspace
such that $V_{\mathcal{P}_{k-1}}(\Omega) \cap \mathcal{P}_{k-1} = \emptyset$
and $\mathcal{P}_{k-1}$ denotes the set of polynomials of degree at most
$k-1$.

The above properties (1) and (2) will be immediately satisfied if the
generalization of \eqref{gateway} holds for the modified bilinear
form:
\begin{equation}
a_k(v,v) \geq c_k \, |v|^2_{H^k(\Omega)} \label{gateway_k}.
\end{equation}
A similar construction of the \emph{Neumann} matrix can be immediately
generalized as follows:
\begin{equation*}
\langle \cN^{(k)}_{HH} \underline{\phi}, \underline{\psi} \rangle := 
a_k(\phi_H^h,\psi_H^h).
\end{equation*}
The low-rank perturbations arise from the following decomposition of
$K_{HH}^{(k)}(m)$:
\begin{equation*}
K_{HH}^{(k)}(m) =  m \, \cN_{HH}^{(k)} + R^{(k)}, \quad 
\left( K_{HH}^{(k)}(m) \right)^{-1} = e_H^{(k)} \eta^{(k)^{-1}} e_H^{(k)^t} + 
\mathcal{O}(m^{-1}), 
\end{equation*}
where $\eta^{(k)} := e_H^{(k)^t} K_{HH}^{(k)} e_H^{(k)}$.
LRP is produced by $e_H^{(k)} \in \mathcal{P}_{k-1}^h$ because the rank is equal
to the cardinality of the basis polynomials in $\mathcal{P}_{k-1}^h$.
\begin{equation*}
\ker \cN_{HH}^{(k)} = \mathcal{P}_{k-1}^h|_{\overline{\Omega}_H}.
\end{equation*}

Due to \eqref{zeroContributionTerm}, $a_2(\cdot,\cdot)$ in
\eqref{ciarlet1} corresponds to the strong formulation $T_2$ exactly.
Let us denote the strong formulation to which $a_k(\cdot, \cdot)$
corresponds by $\hat{T}_k$.  We have $\hat{T}_k = T_k,~k=1,2$ for the
high-contrast diffusion and biharmonic plate equations, respectively:
\begin{eqnarray*}
a_1(v,v) & := & (\nabla v, \alpha_1 \, \nabla v)\\
a_2(v,v) & := & \sigma_2 \, (\nabla^2 v, \alpha_2 \, \nabla^2 v) + 
\alpha_2 \, (1 - \sigma_2) |v|_{H^2(\Omega)}^2
\end{eqnarray*}

However, for general $k$, $a_k(\cdot, \cdot)$ may not correspond to
$T_k$.  In addition, one may need more general boundary conditions if
similar zero contributions in \eqref{zeroContributionTerm} can be
obtained for general $k$.  Further research is needed to see if
such boundary conditions are physical. Currently, it is also unclear
for which applications such general PDEs can be used.  However, there
are interesting invariance theory implications when one employs
bilinear forms corresponding to rotationally invariant functions
compatible to energy definition in \eqref{plateEnergyHessian}.  This
allows a generalization of the energy notion and may be the subject
for future research.  For further information, we list the relevant
bilinear forms that are composed of rotationally invariant functions
derived by the utilization of invariance theory.
\begin{eqnarray*}
a_3(v,v) & := & \sigma_3 \, (\nabla^3 v, \alpha_3 \,\nabla^3 v) + 
\alpha_3 \, (1 - \sigma_3) |v|_{H^3(\Omega)}^2\\
a_4(v,v) & := & \sigma_4 \, (\nabla^4 v, \alpha_4 \,\nabla^4 v) + 
\alpha_4 \, (1 - \sigma_4) |v|_{H^4(\Omega)}^2 + 
\alpha_4 \, \gamma_4 |\nabla^2 v|_{H^2(\Omega)}^2.
\end{eqnarray*}
Note that the above bilinear forms satisfy \eqref{gateway_k}.

\section{NUMERICAL EXPERIMENTS} \label{sec:numerics}

The goal of the numerical experiments is to compare the performance of
the two preconditioners: AGKS and MG. The domain is a unit square
whose coarsest level triangulation consists of $32$ triangles.  We
consider the case of a single highly-bending island located at the
region $[1/4,2/4] \times [1/4,2/4]$ consisting of 2 coarsest level
triangles.  For an extension to the case of multiple disconnected
islands, one can refer to~\cite[Sections 3 and 4]{AGKS:2007}.  The
implementation of HCT and Morley discretizations are based on
Pozrikidis' software provided in \cite{pozrikidis2005_book}.  The
problems sizes of HCT and Morley discretizations are $131$, $451$,
$1667$, $6403$ and $81$, $289$, $1089$, $4225$ for levels $1, 2, 3$
and $4$, respectively.

We denote the norm of the relative residual at iteration $i$ by $rr^{(i)}$:
\begin{equation*}
rr^{(i)} := \frac{\|r^{(i)}\|_2}{\|r^{(0)}\|_2},
\end{equation*}
where $r^{(i)}$ denotes the residual at iteration $i$ with a stopping
criterion of $rr^{(i)} \leq 10^{-7}.$ In Tables \ref{ahs}--\ref{mmg},
preconditioned conjugate gradient iteration count and the average
reduction factor are reported for combinations of preconditioner,
smoother types, and number of smoothing iterations. The average
reduction factor of the residual is defined as:
\begin{equation*}
\left( rr^{(i)}\right )^{1/i}.
\end{equation*} 
We enforce an iteration bound of $60$. If the method
seems to converge slightly beyond this bound, we denote
it by $60^+$, whereas, stalling is denoted by $\infty.$

We use Galerkin variational approach to construct the coarser level
algebraic systems. The multigrid preconditioner MG is derived from the
implementation by Aksoylu, Bond, and Holst~\cite{AkBoHo03}.  We
employ a V(1,1)-cycle with point symmetric Gauss-Seidel (sGS) and
point Gauss-Seidel (GS) smoothers.  A direct solver is used for the
coarsest level.  

By exploiting the fact that $S_\infty$ in \eqref{limitingS} is
only a LRP of $K_{LL}$, we can build robust
preconditioners for $S_\infty$ in \eqref{MainPrec1} via standard
multigrid preconditioners.  \eqref{limitingS} implies that
\begin{equation*}
S_\infty = K_{LL} - v \eta^{-1}v^T,
\end{equation*}
where $v := K_{LH} e_H$. If $M_{LL}$ denotes a
standard multigrid V-cycle for $K_{LL}$, we can construct an
efficient and robust preconditioner $\tilde{S}^{-1}$ for $S_\infty$
using the Sherman-Morrison-Woodbury formula, i.e.
\begin{equation} \label{Sherman-Morrison}
\tilde{S}^{-1} \ := \ M_{LL} \ + \ 
M_{LL} v ~(\eta - v^T M_{LL} v)^{-1} \, v^T M_{LL}.
\end{equation}
Note also that we can precompute and store $M_{LL} v$ during the setup
phase.  This means that we only need to apply the multigrid V-cycle
$M_{LL}$ once per iteration.  Therefore, the following practical
version of preconditioner \eqref{MainPrec1} is used in the
implementation:
\begin{eqnarray}
\tilde{B}_{AGKS} & := &
{ \left[ \begin{array}{cc}  I_{HH} & - K_{HH}^{\infty^\dagger} K_{HL} \\
0 & I_{LL}\end{array}\right]}
{ \left[ \begin{array}{cc}    M_{HH} & 0 \\
0&  \tilde{S}^{-1} \end{array}\right]} 
{ \left[ \begin{array}{cc}  I_{HH} & 0 \\  
- K_{LH}K_{HH}^{\infty^\dagger} & I_{LL}
\end{array}\right]} \label{practicalAGKS}.
\end{eqnarray}
We construct two different multilevel hierarchies for
multigrid preconditioners $M_{HH}$ in \eqref{practicalAGKS} and
$M_{LL}$ in \eqref{Sherman-Morrison} for DOF corresponding to
$\Omega_H$ and $\Omega_L$, respectively.  For prolongation, linear
interpolation is used as in \cite{Braess.D;Peisker.P1987}.
The prolongation matrices $P_{HH}$ and $P_{LL}$ are extracted from the
prolongation matrix for whole domain $\Omega$ in the fashion following
\eqref{2x2blockSys}:
\begin{equation*} \label{2x2blockSysProlongation}
P =  
\left[
\begin{array}{ll}
P_{HH} & P_{HL} \\ P_{LH} & P_{LL}
\end{array}
\right].
\end{equation*}

\input{tablesForPaper2_final}

As emphasized in our preceding paper~\cite{AGKS:2007}, AGKS can be
used purely as an algebraic preconditioner.  Therefore, the standard
multigrid preconditioner constraint that the coarsest level mesh
resolves the boundary of the island is automatically
eliminated. However, for a fair comparison, we enforce the coarsest
level mesh to have that property.

We do not observe convergence improvement when a subdomain deflation
strategy based on the smallest eigenvalues is used as in the diffusion
equation case~\cite{AkYe2009}.  The eigenvectors of the Neumann matrix,
$e_H$ in \eqref{defn:e_H}, cannot approximate the eigenvectors
corresponding to the smallest eigenvalues of $K_{HH}$ which are of
$\mathcal{O}(1)$ (see Figure \ref{fig:spectra}) since the remainder
matrix $R$ in \eqref{K_HH_decomp} is of $\mathcal{O}(10^4)$.
Therefore, a deflation strategy utilizing $e_H$ will not necessarily
guarantee deflation of the smallest eigenvalues of $K_{HH}$ in the
biharmonic case.

We first observe that the Morley discretization provides faster
convergence for both preconditioners. Then, we have the following
results regarding the effect of number of smoothing iterations on the
convergence behaviour.  The convergence of MG heavily depends on the
number of smoothing iterations, i.e., the more the smoothing
iteration, the faster the convergence.  For the HCT discretization,
AGKS requires more than a single smoothing iteration for convergence;
see Tables \ref{ahs} and \ref{ahg}. However, for the Morley
discretization, even with the same minimal number of smoothing
iteration, AGKS leads to convergence; see Tables \ref{ams} and
\ref{amg}.  The choice of 5 smoothing iterations is sufficient for
AGKS to reach $h$-robustness and its peak performance. Hence, we can
conclude that AGKS clearly enjoys $h$-robustness.  In contrast, MG is
not $h$-robust regardless of the $m$ value and the smoothing number;
see Tables \ref{mhs}, \ref{mhg}, \ref{mms}, and \ref{mmg}.  MG is
totally ineffective as the problem size increases for both
discretizations, and more obviously for HCT.

Finally, we report the $m$-robustness results.  The loss of
$m$-robustness of MG can be observed consistently for all $m$ values;
see Tables \ref{mhs}, \ref{mhg}, \ref{mms}, and \ref{mmg}.  The AGKS
preconditioner becomes more effective with increasing $m$ and reaches
its peak performance by maintaining an optimal iteration count for all
$m \geq 10^5$. This indicates that $m \geq 10^5$ corresponds to the
asymptotic regime.  Even increasing the $m$ value from $10^2$ to
$10^3$ reduces the iteration count significantly, a clear sign of
close proximity to the asymptotic regime.  In addition, the AGKS
outperforms MG even for $m=1$. Consequently, for both discretizations,
we infer that AGKS is $m$-robust.




\end{document}

%% file: tablesForPaper2_final.tex
{\setlength\arraycolsep{0.15em}
\begin{table*}
{\footnotesize
\caption{AGKS + HCT + sGS + smooth number 1-5-10}
\label{ahs}      

$$\begin{array}{lrrrrrrrrr}
  \hline \\
  N \backslash m & 10^0 & 10^1 & 10^2 & 10^3 & 10^4 & 10^5 &
  10^7 & 10^9 & 10^{10} \\[1ex]
  \hline\hline
  &&&&\textbf{smooth}&\textbf{number}& =~1~~~~~~~&&& \\[1ex]
  \hline 
  \mathbf{131}& \mathbf{24}, 0.485 & \mathbf{20}, 0.447 & 
  \mathbf{18}, 0.407 & \mathbf{17}, 0.371 & \mathbf{17},
  0.381  & \mathbf{16}, 0.337 & \mathbf{18}, 0.371 & \mathbf{16}, 0.362 &
  \mathbf{17}, 0.384  \\[1ex]
  \mathbf{451} & \mathbf{52}, 0.730 & \mathbf{38}, 0.650  &
  \mathbf{21}, 0.452 & \mathbf{13}, 0.286 & \mathbf{12}, 0.249
  & \mathbf{12}, 0.256 & \mathbf{13}, 0.279 & \mathbf{12}, 0.253 & \mathbf{11}, 
  0.213  \\[1ex]
  \mathbf{1667} & \mathbf{60}^+, 0.857 & \mathbf{60}^+,0.768  & 
  \mathbf{33}, 0.610 & \mathbf{20}, 0.426 & \mathbf{18},
  0.401  & \mathbf{19}, 0.410 & \mathbf{21}, 0.447 & \mathbf{19},
  0.420  & \mathbf{19}, 0.417  \\[1ex]
  \mathbf{6403}  & \boldsymbol{\infty}, 0.972 & \mathbf{60}^+, 0.930 & 
  \mathbf{60}^+, 0.839 & \mathbf{45}, 0.692 & \mathbf{37}, 0.637
  & \mathbf{36}, 0.636 & \mathbf{36}, 0.638 & \mathbf{36}, 0.635 & 
  \mathbf{39}, 0.661   \\[1ex]
  \hline 
  &&&&\textbf{smooth}&\textbf{number}& =~5~~~~~&&& \\[1ex]
  \hline 
  \mathbf{131}& \mathbf{24}, 0.485 & \mathbf{20}, 0.447 & 
  \mathbf{18}, 0.407 & \mathbf{17}, 0.371 & \mathbf{17},
  0.381  & \mathbf{16}, 0.337 & \mathbf{18}, 0.371 & \mathbf{16}, 0.362 &
  \mathbf{17}, 0.384  \\[1ex]
  \mathbf{451} & \mathbf{40}, 0.664 & \mathbf{28},  0.547&
  \mathbf{15}, 0.330 & \mathbf{8}, 0.131 & \mathbf{6}, 0.054
  & \mathbf{6}, 0.023 & \mathbf{4}, 0.014 & \mathbf{4}, 0.016 & \mathbf{4}, 
  0.012 \\[1ex]
  \mathbf{1667} & \mathbf{60}^+, 0.786 & \mathbf{48}, 0.706& 
  \mathbf{24}, 0.490 & \mathbf{12}, 0.258 & \mathbf{8},
  0.091  & \mathbf{6}, 0.058 & \mathbf{5}, 0.035 & \mathbf{5},
  0.026  & \mathbf{5}, 0.024   \\[1ex]
  \mathbf{6403}& \mathbf{60}^+, 0.947 & \mathbf{60}^+, 0.862
  & \mathbf{43}, 0.682  &  \mathbf{21}, 0.427
  & \mathbf{12}, 0.223 & \mathbf{8}, 0.091 & \mathbf{6}, 0.051
  & \mathbf{6}, 0.052& \mathbf{6}, 0.062  \\[1ex]
  \hline 
  &&&&\textbf{smooth}&\textbf{number}&=~10~~~~~&&& \\[1ex]
  \hline 
  \mathbf{131}& \mathbf{24}, 0.485 & \mathbf{20}, 0.447 & 
  \mathbf{18}, 0.407 & \mathbf{17}, 0.371 & \mathbf{17},
  0.381  & \mathbf{16}, 0.337 & \mathbf{18}, 0.371 & \mathbf{16}, 0.362 &
  \mathbf{17}, 0.384  \\[1ex]
  \mathbf{451} & \mathbf{37}, 0.634 & \mathbf{26}, 0.528&
  \mathbf{15}, 0.330 & \mathbf{8}, 0.131 & \mathbf{6}, 0.050
  & \mathbf{6}, 0.017 & \mathbf{4}, 0.010 & \mathbf{3}, 0.004 & \mathbf{3},
  0.003   \\[1ex]
  \mathbf{1667} & \mathbf{60}^+, 0.785 & \mathbf{43}, 0.680&
  \mathbf{20}, 0.442 & \mathbf{12}, 0.213 & \mathbf{8},
  0.080  & \mathbf{6}, 0.030 & \mathbf{4}, 0.004 & \mathbf{4},0.002
  & \mathbf{4}, 0.008   \\[1ex]
  \mathbf{6403}& \mathbf{60}^+,  0.943 & \mathbf{60}^+, 0.861 
  & \mathbf{38}, 0.653  & \mathbf{20}, 0.410
  & \mathbf{10}, 0.177 & \mathbf{8}, 0.090 & \mathbf{5}, 0.028
  &\mathbf{5}, 0.015  &\mathbf{5},0.023  \\[1ex]
  \hline
\end{array}$$
}
\end{table*}

\begin{table*}
{\footnotesize
\caption{AGKS + HCT + GS + smooth number 1-5-10}
\label{ahg}      
$$\begin{array}{lrrrrrrrrrr}
  \hline \\
  N \backslash m & 10^0 & 10^1 & 10^2 & 10^3 & 10^4 & 10^5 &
  10^7 & 10^9 & 10^{10} \\[1ex]
  \hline\hline
  &&&&\textbf{smooth}&\textbf{number}& =~1~~~~~~~&&& \\[1ex]
  \hline
  \mathbf{131}& \mathbf{24}, 0.485 & \mathbf{20}, 0.447 & 
  \mathbf{18}, 0.407 & \mathbf{17}, 0.371 & \mathbf{17},
  0.381  & \mathbf{16}, 0.337 & \mathbf{18}, 0.371 & \mathbf{16}, 0.362 &
  \mathbf{17}, 0.384  \\[1ex]
  \mathbf{451}  & \mathbf{57}, 0.749 & \mathbf{49},0.720 & 
  \mathbf{27}, 0.538 & \mathbf{23}, 0.494 & \mathbf{22},
  0.459  & \mathbf{23}, 0.480 & \mathbf{26}, 0.535 & \mathbf{24},
  0.490  & \mathbf{25}, 0.517  \\[1ex]
  \mathbf{1667}  & \mathbf{60}^+, 0.918 & \mathbf{60}^+, 0.880  &
  \mathbf{60}^+, 0.872 &\mathbf{60}^+, 0.847  &\mathbf{60}^+, 0.853
  & \mathbf{60}^+, 0.820 & \mathbf{60}^+, 0.871 & \mathbf{60}^+, 0.881 
  &\mathbf{60}^+, 0.814  \\[1ex]
  \mathbf{6403} & \boldsymbol{\infty}, 1.001 & \boldsymbol{\infty}, 0.991 & 
  \boldsymbol{\infty}, 0.958  & \boldsymbol{\infty}, 0.953  &
  \boldsymbol{\infty}, 0.964   &\boldsymbol{\infty}, 0.971  &  
  \boldsymbol{\infty}, 0.980 &   \boldsymbol{\infty}, 0.977  &
  \boldsymbol{\infty}, 0.985   \\[1ex]
  \hline
  &&&&\textbf{smooth}&\textbf{number}& =~5~~~~~&&& \\[1ex]
  \hline
  \mathbf{131}& \mathbf{24}, 0.485 & \mathbf{20}, 0.447 & 
  \mathbf{18}, 0.407 & \mathbf{17}, 0.371 & \mathbf{17},
  0.381  & \mathbf{16}, 0.337 & \mathbf{18}, 0.371 & \mathbf{16}, 0.362 &
  \mathbf{17}, 0.384  \\[1ex]
  \mathbf{451} & \mathbf{37}, 0.644 & \mathbf{26}, 0.538&
  \mathbf{16}, 0.339 & \mathbf{8}, 0.133 & \mathbf{6}, 0.053
  & \mathbf{6}, 0.024 & \mathbf{4}, 0.011 & \mathbf{4}, 0.010 & \mathbf{4}, 
  0.014   \\[1ex]
  \mathbf{1667} & \mathbf{60}^+, 0.786 & \mathbf{48}, 0.706 & 
  \mathbf{23}, 0.494 & \mathbf{12}, 0.253 & \mathbf{8},
  0.106  & \mathbf{6}, 0.060& \mathbf{5}, 0.022& \mathbf{5}, 0.022 & 
  \mathbf{5},  0.027  \\[1ex]
  \mathbf{6403} &\mathbf{60}^+, 0.947 & \mathbf{60}^+, 0.887&
  \mathbf{50}, 0.724 &  \mathbf{22}, 0.480 & \mathbf{12}, 0.253 & \mathbf{9},
  0.141 &  \mathbf{10}, 0.185& \mathbf{9}, 0.138 & \mathbf{10},0.163   \\[1ex]
  \hline
  &&&&\textbf{smooth}&\textbf{number}&=~10~~~~~&&& \\[1ex]
  \hline
  \mathbf{131}& \mathbf{24}, 0.485 & \mathbf{20}, 0.447 & 
  \mathbf{18}, 0.407 & \mathbf{17}, 0.371 & \mathbf{17},
  0.381  & \mathbf{16}, 0.337 & \mathbf{18}, 0.371 & \mathbf{16}, 0.362 &
  \mathbf{17}, 0.384  \\[1ex]
  \mathbf{451} & \mathbf{37}, 0.637 & \mathbf{25}, 0.525&
  \mathbf{14}, 0.312 & \mathbf{8}, 0.131 & \mathbf{6}, 0.050
  & \mathbf{6}, 0.016 & \mathbf{4}, 0.002 & \mathbf{3}, 0.004 & \mathbf{3}, 
  0.002  \\[1ex]
  \mathbf{1667} & \mathbf{60}^+, 0.785 & \mathbf{43}, 0.680 & 
  \mathbf{20}, 0.442 & \mathbf{12}, 0.213 & \mathbf{8},
  0.080  & \mathbf{6}, 0.029& \mathbf{4}, 0.005 & \mathbf{4},
  0.002  & \mathbf{4}, 0.006 \\[1ex]
  \mathbf{6403} & \mathbf{60}^+, 0.946 & \mathbf{60}^+, 0.861 & \mathbf{45},
 0.696 & \mathbf{20}, 0.410 & \mathbf{10}, 0.196 & \mathbf{8}, 0.085 &
 \mathbf{6}, 0.052 & \mathbf{6}, 0.033 & \mathbf{6},0.040  \\[1ex]
  \hline
\end{array}$$
}
\end{table*}

\begin{table*}
{\footnotesize
\caption{AGKS + Morley +  sGS + smooth number 1-5-10}
\label{ams}      
$$\begin{array}{lrrrrrrrrrr}
  \hline \\
  N \backslash m & 10^0 & 10^1 & 10^2 & 10^3 & 10^4 & 10^5 &
  10^7 & 10^9 & 10^{10}\\[1ex]
  \hline\hline
  &&&&\textbf{smooth}&\textbf{number}& =~1~~~~~~~&&& \\[1ex]
  \hline
  \mathbf{81}& \mathbf{9}, 0.119 & \mathbf{7}, 0.066 & 
  \mathbf{5}, 0.040 & \mathbf{4}, 0.015 & \mathbf{4},
  0.005  & \mathbf{4}, 0.002 & \mathbf{2}, 0.0003 & \mathbf{2},
  0.0003 & \mathbf{3}, 0.0002\\[1ex]
  \mathbf{289} & \mathbf{15}, 0.316 & \mathbf{10}, 0.189&
  \mathbf{8}, 0.080 & \mathbf{6}, 0.030 & \mathbf{4}, 0.012
  & \mathbf{4}, 0.004 & \mathbf{3}, 0.004 & \mathbf{3}, 0.001 & \mathbf{3}, 
  0.0004   \\[1ex]
  \mathbf{1089} & \mathbf{25}, 0.519 & \mathbf{17}, 0.367   & 
  \mathbf{10}, 0.182 & \mathbf{8}, 0.074 & \mathbf{6},
  0.027  & \mathbf{4}, 0.015& \mathbf{4}, 0.005 & \mathbf{4},
  0.003  & \mathbf{4}, 0.003\\[1ex]
  \mathbf{4225} & \mathbf{52}, 0.733 &  \mathbf{34}, 0.619
  &  \mathbf{18}, 0.358 &  \mathbf{10}, 0.173
  & \mathbf{7}, 0.084 &  \mathbf{6}, 0.044 & \mathbf{6}, 0.035
  &  \mathbf{6}, 0.043 & \mathbf{6}, 0.036  \\[1ex]
  \hline
  &&&&\textbf{smooth}&\textbf{number}& =~5~~~~~&&& \\[1ex]
  \hline
  \mathbf{81}& \mathbf{9}, 0.119 & \mathbf{7}, 0.066 & 
  \mathbf{5}, 0.040 & \mathbf{4}, 0.015 & \mathbf{4},
  0.005  & \mathbf{4}, 0.002 & \mathbf{2}, 0.0003 & \mathbf{2},
  0.0003 & \mathbf{3}, 0.0002\\[1ex]
  \mathbf{289} & \mathbf{15}, 0.316 & \mathbf{10}, 0.189&
  \mathbf{8}, 0.080 & \mathbf{6}, 0.031 & \mathbf{4}, 0.012
  & \mathbf{4}, 0.004 & \mathbf{3}, 0.004 & \mathbf{2}, 0.0003 & 
  \mathbf{3}, 0.0004   \\[1ex]
  \mathbf{1089} & \mathbf{25}, 0.514 & \mathbf{17}, 0.363   & 
  \mathbf{10}, 0.181 & \mathbf{8}, 0.074 & \mathbf{6},
  0.024  & \mathbf{4}, 0.009& \mathbf{4}, 0.001 & \mathbf{3},
  0.002  & \mathbf{3}, 0.001 \\[1ex]
  \mathbf{4225} & \mathbf{46}, 0.698  & \mathbf{27}, 0.546 & \mathbf{16},
  0.315 & \mathbf{10}, 0.152 & \mathbf{6}, 0.057 & \mathbf{6}, 0.018 & 
  \mathbf{4},  0.003 & \mathbf{4}, 0.002& \mathbf{3}, 0.004  \\[1ex]
  \hline
  &&&&\textbf{smooth}&\textbf{number}&=~10~~~~~&&& \\[1ex]
  \hline
  \mathbf{81}& \mathbf{9}, 0.119 & \mathbf{7}, 0.066 & 
  \mathbf{5}, 0.040 & \mathbf{4}, 0.015 & \mathbf{4},
  0.005  & \mathbf{4}, 0.002 & \mathbf{2}, 0.0003 & \mathbf{2},
  0.0003 & \mathbf{3}, 0.0002\\[1ex]
  \mathbf{289} & \mathbf{15}, 0.316 & \mathbf{10}, 0.189&
  \mathbf{8}, 0.080 & \mathbf{6}, 0.031 & \mathbf{4}, 0.012
  & \mathbf{4}, 0.004 & \mathbf{3}, 0.004 & \mathbf{2}, 0.0002 & 
  \mathbf{3}, 0.0004   \\[1ex]
  \mathbf{1089} & \mathbf{25}, 0.514 & \mathbf{17}, 0.363   & 
  \mathbf{10}, 0.181 & \mathbf{8}, 0.074 & \mathbf{6},
  0.024  & \mathbf{4}, 0.009& \mathbf{4}, 0.001 & \mathbf{3},
  0.002  & \mathbf{3}, 0.001\\[1ex]
  \mathbf{4225} &\mathbf{46}, 0.698 &\mathbf{27}, 0.546
  & \mathbf{16}, 0.315 &\mathbf{10}, 0.151 & \mathbf{6}, 0.057 & \mathbf{6}, 
  0.018 & \mathbf{4}, 0.003 & \mathbf{4}, 0.002& \mathbf{4}, 0.004 \\[1ex]
  \hline
\end{array}$$
}
\end{table*}

\begin{table*}
{\footnotesize
\caption{AGKS + Morley + GS + smooth number 1-5-10}
\label{amg}      
$$\begin{array}{lrrrrrrrrrrr}
  \hline \\
  N \backslash m & 10^0 & 10^1 & 10^2 & 10^3 & 10^4 & 10^5 &
  10^7 & 10^9 & 10^{10}  \\[1ex]
  \hline\hline
  &&&&\textbf{smooth}&\textbf{number}& =~1~~~~~~~&&& \\[1ex]
  \hline
  \mathbf{81}& \mathbf{9}, 0.119 & \mathbf{7}, 0.066 & 
  \mathbf{5}, 0.040 & \mathbf{4}, 0.015 & \mathbf{4},
  0.005  & \mathbf{4}, 0.002 & \mathbf{2}, 0.0003 & \mathbf{2},
  0.0003 & \mathbf{3}, 0.0002\\[1ex]
  \mathbf{289} & \mathbf{15}, 0.329 & \mathbf{10}, 0.189&
  \mathbf{8}, 0.080 & \mathbf{6}, 0.031 & \mathbf{4}, 0.012
  & \mathbf{4}, 0.006 & \mathbf{3}, 0.005 & \mathbf{3}, 0.003 & 
  \mathbf{3}, 0.003   \\[1ex]
  \mathbf{1089} & \mathbf{28}, 0.550 & \mathbf{19}, 0.402  & 
  \mathbf{10}, 0.192 & \mathbf{8}, 0.085 & \mathbf{6},
  0.043  & \mathbf{5}, 0.040& \mathbf{5}, 0.037 & \mathbf{6},
  0.030 & \mathbf{5}, 0.039  \\[1ex]
  \mathbf{4225}  & \mathbf{59}, 0.760 & \mathbf{38}, 0.651 & 
  \mathbf{20}, 0.433 & \mathbf{12}, 0.249 & \mathbf{11},
  0.204  & \mathbf{11}, 0.214& \mathbf{11}, 0.208 & \mathbf{11},
  0.210 & \mathbf{11}, 0.206 \\[1ex]
  \hline
  &&&&\textbf{smooth}&\textbf{number}& =~5~~~~~&&& \\[1ex]
  \hline
  \mathbf{81}& \mathbf{9}, 0.119 & \mathbf{7}, 0.066 & 
  \mathbf{5}, 0.040 & \mathbf{4}, 0.015 & \mathbf{4},
  0.005  & \mathbf{4}, 0.002 & \mathbf{2}, 0.0003 & \mathbf{2},
  0.0003 & \mathbf{3}, 0.0002\\[1ex]
  \mathbf{289} & \mathbf{15}, 0.316 & \mathbf{10}, 0.189&
  \mathbf{8}, 0.080 & \mathbf{6}, 0.031 & \mathbf{4}, 0.012
  & \mathbf{4}, 0.004 & \mathbf{3}, 0.004 & \mathbf{2}, 0.0002 & 
  \mathbf{3}, 0.0004   \\[1ex]
  \mathbf{1089} & \mathbf{25}, 0.514 & \mathbf{17}, 0.363&
  \mathbf{10}, 0.181 & \mathbf{8}, 0.074 & \mathbf{6},
  0.024  & \mathbf{4}, 0.009& \mathbf{4}, 0.001  & \mathbf{3}, 0.002 & 
  \mathbf{3},  0.001 \\[1ex]
  \mathbf{4225} &\mathbf{46}, 0.698 &\mathbf{27}, 0.0546
  & \mathbf{16}, 0.315 &\mathbf{10}, 0.152 & \mathbf{6}, 0.057 & \mathbf{6},
  0.018 & \mathbf{4}, 0.003 & \mathbf{4}, 0.002& \mathbf{3}, 0.004 \\[1ex]
  \hline
  &&&&\textbf{smooth}&\textbf{number}&=~10~~~~~&&& \\[1ex]
  \hline
  \mathbf{81}& \mathbf{9}, 0.119 & \mathbf{7}, 0.066 & 
  \mathbf{5}, 0.040 & \mathbf{4}, 0.015 & \mathbf{4},
  0.005  & \mathbf{4}, 0.002 & \mathbf{2}, 0.0003 & \mathbf{2},
  0.0003 & \mathbf{3}, 0.0002\\[1ex]
  \mathbf{289} & \mathbf{15}, 0.316 & \mathbf{10}, 0.189&
  \mathbf{8}, 0.080 & \mathbf{6}, 0.031 & \mathbf{4}, 0.012
  & \mathbf{4}, 0.004 & \mathbf{3}, 0.004 & \mathbf{2}, 0.0001 & 
  \mathbf{3}, 0.0004   \\[1ex]
  \mathbf{1089} & \mathbf{25}, 0.514 & \mathbf{17}, 0.363&
  \mathbf{10}, 0.181 & \mathbf{8}, 0.074 & \mathbf{6},
  0.024  & \mathbf{4}, 0.009& \mathbf{4}, 0.001  & \mathbf{3}, 0.002 & 
  \mathbf{3},  0.001\\[1ex]
  \mathbf{4225} &\mathbf{46}, 0.698 &\mathbf{27}, 0.0546
  & \mathbf{16}, 0.315 &\mathbf{10}, 0.152 & \mathbf{6}, 0.057 & \mathbf{6},
  0.018 & \mathbf{4}, 0.003 & \mathbf{4}, 0.002& \mathbf{3}, 0.003 \\[1ex]
  \hline
\end{array}$$
}
\end{table*}

\begin{table*}
{\footnotesize
\caption{MG + HCT + sGS + smooth number 1-5-10}
\label{mhs}      
$$\begin{array}{lrrrrrrrrrrrr}
  \hline \\
  N \backslash m & 10^0 & 10^1 & 10^2 & 10^4 & 10^5 & 10^6 & 10^7 &
  10^8 & 10^9\\[1ex]
  \hline\hline \\
  \hline
  &&&&\textbf{smooth}&\textbf{number}& =~1~~~~~&&& \\[1ex]
  \hline
  \mathbf{131}& \mathbf{60}^+, 0.885 & \mathbf{60}^+, 0.898  & \mathbf{60}^+, 0.932 &
  \boldsymbol{\infty}, 0.988 &  \boldsymbol{\infty},
  0.997 & \boldsymbol{\infty}, 1.075 & \boldsymbol{\infty}, 1.089 &
  \boldsymbol{\infty}, 1.065 & \boldsymbol{\infty}, 1.137 \\[1ex]
  \mathbf{451}  &  \boldsymbol{\infty}, 0.963 & \boldsymbol{\infty}, 0.987 &
\boldsymbol{\infty}   1.014 &  \boldsymbol{\infty}, 1.050&\boldsymbol{\infty},
  1.086 & \boldsymbol{\infty}, 1.106& \boldsymbol{\infty}, 1.172& \boldsymbol{\infty},
  1.081&\boldsymbol{\infty}, 1.091\\[1ex]
  \mathbf{1667}  &  \boldsymbol{\infty}, 0.985& \boldsymbol{\infty}, 1.015&
\boldsymbol{\infty}, 1.044&  \boldsymbol{\infty}, 1.062& \boldsymbol{\infty},
  1.122 & \boldsymbol{\infty}, 1.109 & \boldsymbol{\infty}, 1.142& \boldsymbol{\infty},
  1.170 &\boldsymbol{\infty}, 1.124\\[1ex]
  \mathbf{6403}  &  \boldsymbol{\infty}, 1.025 & \boldsymbol{\infty}, 1.040 & 
  \boldsymbol{\infty}, 1.057 & \boldsymbol{\infty}, 1.125 & \boldsymbol{\infty}, 1.145
  & \boldsymbol{\infty}, 1.130 & \boldsymbol{\infty}, 1.171 & \boldsymbol{\infty}, 1.112
  & \boldsymbol{\infty}, 1.187  \\[1ex]
  \hline
 &&&&\textbf{smooth}&\textbf{number}& =~5~~~~~&&& \\[1ex]
  \hline
  \mathbf{131}& \mathbf{60}^+, 0.885 & \mathbf{60}^+, 0.898  & \mathbf{60}^+, 0.932
  &  \boldsymbol{\infty}, 0.988 &  
  \boldsymbol{\infty}, 0.997 & \boldsymbol{\infty}, 1.075 & \boldsymbol{\infty},
  1.089 & \boldsymbol{\infty}, 1.065 & \boldsymbol{\infty}, 1.137 \\[1ex]
  \mathbf{451} & \mathbf{60}^+, 0.761 & \mathbf{60}^+, 0.829 & \mathbf{60}^+, 0.920 &  
  \boldsymbol{\infty}, 1.070 & \boldsymbol{\infty}, 1.084 &
  \boldsymbol{\infty}, 1.120& \boldsymbol{\infty}, 1.174 &  \boldsymbol{\infty}, 1.118& 
  \boldsymbol{\infty}, 1.166 \\[1ex]
  \mathbf{1667}& \mathbf{60}^+, 0.854 & \mathbf{60}^+, 0.923 & \boldsymbol{\infty}, 0.999
 & \boldsymbol{\infty}, 1.038 & \boldsymbol{\infty}, 1.0037 &
  \boldsymbol{\infty}, 1.0085 & \boldsymbol{\infty}, 1.134 &  \boldsymbol{\infty}, 1.154
  & \boldsymbol{\infty}, 1.208\\[1ex]
  \mathbf{6403}& \mathbf{60}^+, 0.931 & \boldsymbol{\infty}, 0.979 & \boldsymbol{\infty},
0.998 &  \boldsymbol{\infty},1.012 & \boldsymbol{\infty}, 1.023 &
  \boldsymbol{\infty}, 1.058 & \boldsymbol{\infty}, 1.041 &  \boldsymbol{\infty}, 1.063
  & \boldsymbol{\infty}, 1.099 \\[1ex]
  \hline
  &&&&\textbf{smooth}&\textbf{number}& =~10~~~~~&&& \\[1ex]
  \hline
  \mathbf{131}& \mathbf{60}^+, 0.885 & \mathbf{60}^+, 0.898  & \mathbf{60}^+, 0.932
  &  \boldsymbol{\infty}, 0.988 &  
  \boldsymbol{\infty}, 0.997 & \boldsymbol{\infty}, 1.075 & \boldsymbol{\infty},
  1.089 & \boldsymbol{\infty}, 1.065 & \boldsymbol{\infty}, 1.137 \\[1ex]
  \mathbf{451} & \mathbf{48}, 0.660 & \mathbf{53}, 0.701 & \mathbf{60}^+, 0.825 &  
  \boldsymbol{\infty}, 0.955& \boldsymbol{\infty}, 1.032 &
  \boldsymbol{\infty}, 1.115& \boldsymbol{\infty}, 1.179&  \boldsymbol{\infty}, 1.200& 
  \boldsymbol{\infty}, 1.196 \\[1ex]
  \mathbf{1667}& \mathbf{40}, 0.624 & \mathbf{49}, 0.680 & \mathbf{60}^+, 0.797 &  
  \boldsymbol{\infty}, 1.001& \boldsymbol{\infty}, 1.088 &
  \boldsymbol{\infty}, 1.035& \boldsymbol{\infty}, 1.064 &  \boldsymbol{\infty}, 1.052
  & \boldsymbol{\infty}, 1.095\\[1ex]
  \mathbf{6403}& \mathbf{60}^+, 0.890 & \mathbf{60}^+, 0.929 & \boldsymbol{\infty},
0.972 &  \boldsymbol{\infty},1.049 & \boldsymbol{\infty}, 1.017 &
  \boldsymbol{\infty}, 1.052& \boldsymbol{\infty}, 1.051 &  \boldsymbol{\infty}, 1.134
  & \boldsymbol{\infty}, 1.170 \\[1ex]
  \hline
  
\end{array}$$
}
\end{table*}

\begin{table*}
{\footnotesize
\caption{MG + HCT + GS + smooth number 1-5-10}
\label{mhg}      
$$\begin{array}{lrrrrrrrrr}
  \hline \\
  N \backslash m & 10^0 & 10^1 & 10^2  & 10^4 & 10^5 & 10^6 & 10^7 &
  10^8 & 10^9 \\[1ex]
  \hline\hline \\
  \hline
  &&&&\textbf{smooth}&\textbf{number}& =~1~~~~~&&& \\[1ex]
  \hline
  \mathbf{131}& \mathbf{60}^+, 0.885 & \mathbf{60}^+, 0.898  & \mathbf{60}^+, 0.932 & 
  \boldsymbol{\infty}, 0.988 &  \boldsymbol{\infty},
  0.997 & \boldsymbol{\infty}, 1.075 & \boldsymbol{\infty}, 1.089 &
  \boldsymbol{\infty}, 1.065 & \boldsymbol{\infty}, 1.137 \\[1ex]
  \mathbf{451}  & \boldsymbol{\infty}, 0.995 & \boldsymbol{\infty}, 1.016 & 
  \boldsymbol{\infty}, 1.029 &  \boldsymbol{\infty}, 
  1.041 & \boldsymbol{\infty}, 1.107  & \boldsymbol{\infty}, 1.114
  & \boldsymbol{\infty}, 1.128 & \boldsymbol{\infty},1.184  & \boldsymbol{\infty}, 
  1.205 \\[1ex]
  \mathbf{1667}   & \boldsymbol{\infty}, 1.033 & \boldsymbol{\infty}, 1.034 & 
  \boldsymbol{\infty}, 1.042&  \boldsymbol{\infty}, 1.077
  & \boldsymbol{\infty}, 1.115 & \boldsymbol{\infty}, 1.155  & \boldsymbol{\infty}, 
  1.068 & \boldsymbol{\infty}, 1.068 & \boldsymbol{\infty}, 1.079   \\[1ex]
  \mathbf{6403}  &  \boldsymbol{\infty}, 1.055 & \boldsymbol{\infty}, 1.052&
  \boldsymbol{\infty}, 1.057&  \boldsymbol{\infty}, 1.070 & \boldsymbol{\infty},
  1.146& \boldsymbol{\infty}, 1.141& \boldsymbol{\infty}, 1.146 & \boldsymbol{\infty},
  1.115& \boldsymbol{\infty}, 1.160\\[1ex]
  \hline
  &&&&\textbf{smooth}&\textbf{number}& =~5~~~~~&&& \\[1ex]
  \hline
  \mathbf{131}& \mathbf{60}^+, 0.885 & \mathbf{60}^+, 0.898  & \mathbf{60}^+, 0.932 &
  \boldsymbol{\infty}, 0.988 &  \boldsymbol{\infty},
  0.997 & \boldsymbol{\infty}, 1.075 & \boldsymbol{\infty}, 1.089 & 
  \boldsymbol{\infty}, 1.065 & \boldsymbol{\infty}, 1.137 \\[1ex]
  \mathbf{451} & \mathbf{59}, 0.760 & \mathbf{60}^+, 0.828 & \mathbf{60}^+, 0.947  & 
  \boldsymbol{\infty}, 1.043& \boldsymbol{\infty}, 1.061 & \boldsymbol{\infty}
  , 1.121 & \boldsymbol{\infty}, 1.149 & \boldsymbol{\infty}, 1.176 &
  \boldsymbol{\infty}, 1.189  \\[1ex]
  \mathbf{1667} & \mathbf{60}^+, 0.857 & \mathbf{60}^+, 0.904 &  \boldsymbol{\infty}, 
  1.003 & \boldsymbol{\infty}, 1.033 & \boldsymbol{\infty},
  1.056& \boldsymbol{\infty}, 1.101 & \boldsymbol{\infty}, 1.116 & 
  \boldsymbol{\infty}, 1.152& \boldsymbol{\infty}, 1.176 \\[1ex]
  \mathbf{6403} & \boldsymbol{\infty}, 0.961 &  \boldsymbol{\infty}, 1.003 &\boldsymbol{\infty}, 1.037 & \boldsymbol{\infty}, 1.072 & \boldsymbol{\infty}, 1.084
  & \boldsymbol{\infty}, 1.103  & \boldsymbol{\infty}, 1.105 & \boldsymbol{\infty},
  1.115  & \boldsymbol{\infty}, 1.122 \\[1ex]
  \hline
  &&&&\textbf{smooth}&\textbf{number}& =~10~~~~~&&& \\[1ex]
  \hline
  \mathbf{131}& \mathbf{60}^+, 0.885 & \mathbf{60}^+, 0.898  & \mathbf{60}^+, 
  0.932 &
  \boldsymbol{\infty}, 0.988 &  \boldsymbol{\infty},
  0.997 & \boldsymbol{\infty}, 1.075 & \boldsymbol{\infty}, 1.089 & 
  \boldsymbol{\infty}, 1.065 & \boldsymbol{\infty}, 1.137 \\[1ex]
  \mathbf{451} & \mathbf{37}, 0.632 & \mathbf{42}, 0.667& \mathbf{60}^+, 0.778& 
  \boldsymbol{\infty}, 1.060& \boldsymbol{\infty}, 1.081& \boldsymbol{\infty}
  , 1.103 & \boldsymbol{\infty}, 1.161 & \boldsymbol{\infty}, 1.194&
  \boldsymbol{\infty}, 1.200  \\[1ex]
  \mathbf{1667} & \mathbf{47}, 0.706 & \mathbf{57}, 0.753& \mathbf{60}^+, 0.853 &
  \boldsymbol{\infty}, 1.048 & \boldsymbol{\infty},
  1.022& \boldsymbol{\infty}, 1.040 & \boldsymbol{\infty}, 1.072 & 
  \boldsymbol{\infty}, 1.114& \boldsymbol{\infty}, 1.149 \\[1ex]
  \mathbf{6403} & \mathbf{60}^+, 0.924 & \mathbf{60}^+, 0.949 
  &\boldsymbol{\infty}, 1.008& \boldsymbol{\infty}, 1.028& \boldsymbol{\infty}, 
  1.033  & \boldsymbol{\infty}, 1.048 & \boldsymbol{\infty}, 1.037& 
  \boldsymbol{\infty},1.052  & \boldsymbol{\infty}, 1.069 \\[1ex]
  \hline
\end{array}$$
}
\end{table*}

\begin{table*}
{\footnotesize
\caption{MG + Morley + sGS + smooth number 1-5-10}
\label{mms}      
$$\begin{array}{lrrrrrrrrr}
  \hline \\
  N \backslash m & 10^0 & 10^1 & 10^2  & 10^4 & 10^5 & 10^6 & 10^7 &
  10^8 & 10^9 \\[1ex]
  \hline\hline \\
  \hline
  &&&&\textbf{smooth}&\textbf{number}& =~1~~~~~&&& \\[1ex]
  \hline
  \mathbf{81}& \mathbf{38}, 0.652& \mathbf{45}, 0.694 & \mathbf{54}, 0.741 & 
  \mathbf{60}^+, 0.841 & \mathbf{60}^+, 0.924  & \mathbf{60}^+, 0.921
  & \boldsymbol{\infty}, 1.001 & \boldsymbol{\infty}, 1.045 & \boldsymbol{\infty},
  0.959 \\[1ex]
  \mathbf{289} & \mathbf{37}, 0.638& \mathbf{45}, 0.679 & \mathbf{54}, 0.736 &
  \boldsymbol{\infty}, 1.014 & \boldsymbol{\infty}, 1.034 & 
  \boldsymbol{\infty}, 1.021& \boldsymbol{\infty}, 1.031 & \boldsymbol{\infty},
  1.098  & \boldsymbol{\infty},0.021 \\[1ex]
  \mathbf{1089} & \mathbf{50}, 0.724 & \mathbf{60}, 0.766& \mathbf{60}^+,
  0.869 &   \boldsymbol{\infty}, 1.036 & 
  \boldsymbol{\infty}, 1.001& \boldsymbol{\infty}, 1.003 & \boldsymbol{\infty},
  1.002  & \boldsymbol{\infty},1.116  & \boldsymbol{\infty},1.055 \\[1ex]
  \mathbf{4225} & \mathbf{60}^+, 0.877 &  \boldsymbol{\infty}, 1.009& 
  \boldsymbol{\infty}, 
  1.021&  \boldsymbol{\infty}, 1.021 & 
  \boldsymbol{\infty}, 1.061& \boldsymbol{\infty}, 1.057 & \boldsymbol{\infty},
  1.120  & \boldsymbol{\infty},1.151  & \boldsymbol{\infty},1.163 \\[1ex]
  \hline
  &&&&\textbf{smooth}&\textbf{number}& =~5~~~~~&&& \\[1ex]
  \hline
  \mathbf{81}& \mathbf{38}, 0.652& \mathbf{45}, 0.694 & \mathbf{54}, 0.741 & 
  \mathbf{60}^+, 0.841 & \mathbf{60}^+, 0.924  & \mathbf{60}^+, 0.921
  & \boldsymbol{\infty}, 1.001 & \boldsymbol{\infty}, 1.045 & \boldsymbol{\infty},
  0.959 \\[1ex]
  \mathbf{289}  & \mathbf{13}, 0.242 & \mathbf{13}, 0.282 & \mathbf{17}, 0.355 & 
  \mathbf{26}, 0.526 & \mathbf{30}, 0.583 & \mathbf{43},
  0.686  & \mathbf{60}^+, 0.830 &  \boldsymbol{\infty}, 1.027 & \boldsymbol{\infty},
  1.148 \\[1ex]
  \mathbf{1089} & \mathbf{13}, 0.270 & \mathbf{16}, 0.362& \mathbf{20}, 0.407 & 
  \mathbf{31}, 0.585 & \mathbf{46}, 0.691& \mathbf{60}^+,
  0.817  & \boldsymbol{\infty}, 1.025& \boldsymbol{\infty}, 1.005 & \boldsymbol{\infty},
  1.055\\[1ex]
  \mathbf{4225} & \mathbf{34}, 0.621 & \mathbf{41}, 0.668 & \mathbf{51}, 0.726 & 
   \mathbf{60}^+, 0.793& \mathbf{60}^+,
  0.906 & \boldsymbol{\infty}, 1.095 & \boldsymbol{\infty}, 1.027 & 
  \boldsymbol{\infty}, 1.081 & \boldsymbol{\infty}, 1.031 \\[1ex]
  \hline
  &&&&\textbf{smooth}&\textbf{number}& =~10~~~~~&&& \\[1ex]
  \hline
  \mathbf{81}& \mathbf{38}, 0.652& \mathbf{45}, 0.694 & \mathbf{54}, 0.741 &
  \mathbf{60}^+, 0.841 & \mathbf{60}^+, 0.924  & \mathbf{60}^+, 0.921
  & \boldsymbol{\infty}, 1.001 & \boldsymbol{\infty}, 1.045 & \boldsymbol{\infty},
  0.959 \\[1ex]
  \mathbf{289}  & \mathbf{9}, 0.111 & \mathbf{10}, 0.135 & \mathbf{11}, 0.207 & 
  \mathbf{14}, 0.307& \mathbf{18}, 0.395& \mathbf{23},
  0.465  & \mathbf{27}, 0.526& \mathbf{43}, 0.680& \mathbf{60}^+, 0.800  \\[1ex]
  \mathbf{1089} & \mathbf{12}, 0.219 & \mathbf{14}, 0.290& \mathbf{17}, 0.380 & 
  \mathbf{21}, 0.460& \mathbf{26}, 0.477& \mathbf{32},
  0.571  & \mathbf{45}, 0.691& \mathbf{60}^+, 0.772& \boldsymbol{\infty}, 
  0.977\\[1ex]
  \mathbf{4225} & \mathbf{31}, 0.593 & \mathbf{36}, 0.632& \mathbf{44}, 0.682 & 
  \mathbf{55}, 0.742& \mathbf{60}^+, 0.814& \mathbf{60}^+,
  0.900 & \boldsymbol{\infty}, 1.038& \boldsymbol{\infty}, 1.125& 
  \boldsymbol{\infty},
  1.061\\[1ex]
  \hline
\end{array}$$
}
\end{table*}

\begin{table*}
{\footnotesize
\caption{MG + Morley +  GS + smooth number 1-5-10}
\label{mmg}      
$$\begin{array}{lrrrrrrrrr}
  \hline \\
  N \backslash m & 10^0 & 10^1 & 10^2 & 10^4 & 10^5 & 10^6 & 10^7 &
  10^8 & 10^9 \\[1ex]
  \hline\hline \\
  \hline
  &&&&\textbf{smooth}&\textbf{number}& =~1~~~~~&&& \\[1ex]
  \hline
  \mathbf{81}& \mathbf{38}, 0.652& \mathbf{45}, 0.694 & \mathbf{54}, 0.741  &
  \mathbf{60}^+, 0.841 & \mathbf{60}^+, 0.924  & \mathbf{60}^+, 0.921
  & \boldsymbol{\infty}, 1.001 & \boldsymbol{\infty}, 1.045 & \boldsymbol{\infty},
  0.959 \\[1ex]
  \mathbf{289}  & \mathbf{52}, 0.724 & \mathbf{60}^+, 0.807 & \boldsymbol{\infty},
  0.955& \boldsymbol{\infty}, 0.989 & 
  \boldsymbol{\infty}, 0.982& \boldsymbol{\infty}, 1.043 & \boldsymbol{\infty},
  0.990  & \boldsymbol{\infty},1.027  & \boldsymbol{\infty},1.020 \\[1ex]
  \mathbf{1089} & \mathbf{60}^+, 0.860 & \mathbf{60}^+, 0.894 & 
  \boldsymbol{\infty}, 0.996& 
  \boldsymbol{\infty}, 0.989 & 
  \boldsymbol{\infty}, 1.047& \boldsymbol{\infty}, 1.091 & \boldsymbol{\infty},
  1.021  & \boldsymbol{\infty},1.036  & \boldsymbol{\infty},1.173 \\[1ex]
  \mathbf{4225}& \boldsymbol{\infty}, 0.972& \boldsymbol{\infty}, 1.011 & 
  \boldsymbol{\infty},
  1.020 & \boldsymbol{\infty}, 1.066 & 
  \boldsymbol{\infty}, 1.058& \boldsymbol{\infty}, 1.129 & \boldsymbol{\infty},
  1.134  & \boldsymbol{\infty},1.145  & \boldsymbol{\infty},1.164 \\[1ex]
  \hline
  &&&&\textbf{smooth}&\textbf{number}& =~5~~~~~&&& \\[1ex]
  \hline
  \mathbf{81}& \mathbf{38}, 0.652& \mathbf{45}, 0.694 & \mathbf{54}, 0.741 &
  \mathbf{60}^+, 0.841 & \mathbf{60}^+, 0.924  & \mathbf{60}^+, 0.921
  & \boldsymbol{\infty}, 1.001 & \boldsymbol{\infty}, 1.045 & \boldsymbol{\infty},
  0.959 \\[1ex]
  \mathbf{289} & \mathbf{14}, 0.243 & \mathbf{16}, 0.284 & \mathbf{18}, 0.332 & 
  \mathbf{31}, 0.547 & \mathbf{38}, 0.646 & \mathbf{60}^+,
  0.826 &  \boldsymbol{\infty}, 1.037 & \boldsymbol{\infty}, 1.082 & 
  \boldsymbol{\infty}, 1.085 \\[1ex]
  \mathbf{1089} & \mathbf{16}, 0.364 & \mathbf{21}, 0.441& \mathbf{27}, 0.517 & 
  \mathbf{45}, 0.699 & \mathbf{60}^+, 0.774 &  \boldsymbol{\infty}, 1.014 &  
  \boldsymbol{\infty}, 1.042 &  \boldsymbol{\infty}, 1.020 &\boldsymbol{\infty},
  1.038  \\[1ex]
  \mathbf{4225}& \mathbf{39}, 0.652 & \mathbf{50}, 0.718 & \mathbf{60}^+, 0.765 &
  \mathbf{60}^+, 0.844 & \mathbf{60}^+, 0.942  & \boldsymbol{\infty}, 1.073
  & \boldsymbol{\infty}, 1.092 & \boldsymbol{\infty}, 1.107 & \boldsymbol{\infty},
  01.123 \\[1ex]
  \hline
  &&&&\textbf{smooth}&\textbf{number}& =~10~~~~~&&& \\[1ex]
  \hline
  \mathbf{81}& \mathbf{38}, 0.652& \mathbf{45}, 0.694 & \mathbf{54}, 0.741 & 
  \mathbf{60}^+, 0.841 & \mathbf{60}^+, 0.924  & \mathbf{60}^+, 0.921
  & \boldsymbol{\infty}, 1.001 & \boldsymbol{\infty}, 1.045 & \boldsymbol{\infty},
  0.959 \\[1ex]
  \mathbf{289} & \mathbf{10}, 0.104 & \mathbf{10}, 0.186& \mathbf{12}, 0.246 & 
  \mathbf{16}, 0.274& \mathbf{21}, 0.382& \mathbf{26},
  0.537  & \mathbf{38}, 0.626& \mathbf{45}, 0.697& \mathbf{60}^+, 0.870  \\[1ex]
  \mathbf{1089} & \mathbf{15}, 0.289 & \mathbf{16}, 0.362& \mathbf{20}, 0.405 & 
  \mathbf{25}, 0.522& \mathbf{30}, 0.519& \mathbf{36},
  0.607  & \mathbf{47}, 0.702& \mathbf{60}^+, 0.904 &\boldsymbol{\infty},
  1.009  \\[1ex]
  \mathbf{4225}& \mathbf{36}, 0.635& \mathbf{42}, 0.678 & \mathbf{51}, 0.721 & 
  \mathbf{60}^+, 0.830 & \mathbf{60}^+, 0.919  & \boldsymbol{\infty}, 1.027
  & \boldsymbol{\infty}, 1.048 & \boldsymbol{\infty}, 1.003 & \boldsymbol{\infty},
  01.114 \\[1ex]
  \hline
\end{array}$$
}
\end{table*}
}

%% file: paper2_biharmonic.bbl
\begin{thebibliography}{10}
\bibitem{AkBe2008}
{\sc B.~Aksoylu and H.~R. Beyer}, {\em \href{http://arxiv.org/pdf/0810.3427v1}
  {Results on the diffusion equation with rough coefficients}}.
\newblock Submitted to SIAM J. Math. Anal., 2008.

\bibitem{AkBe2009}
\leavevmode\vrule height 2pt depth -1.6pt width 23pt, {\em On the
  characterization of the asymptotic cases of the diffusion equation with rough
  coefficients and applications to preconditioning}, Numer. Funct. Anal.
  Optim., 30 (2009), pp.~405--420.

\bibitem{AkBoHo03}
{\sc B.~Aksoylu, S.~Bond, and M.~Holst}, {\em
  \href{http://siamdl.aip.org/getpdf/servlet/GetPDFServlet?filetype=pdf&id=SJO%
CE3000025000002000478000001&idtype=cvips}{An Odyssey into Local Refinement and
  Multilevel Preconditioning III: Implementation and Numerical Experiments}},
  SIAM J. Sci. Comput., 25 (2003), pp.~478--498.

\bibitem{AGKS:2007}
{\sc B.~Aksoylu, I.~G. Graham, H.~Klie, and R.~Scheichl}, {\em
  \href{http://www.springerlink.com/content/e9633h7441117836/fulltext.pdf}
  {Towards a rigorously justified algebraic preconditioner for high-contrast
  diffusion problems}}, Comput. Vis. Sci., 11 (2008), pp.~319--331.
\newblock doi:10.1007/s00791-008-0105-1.

\bibitem{AkKl:2007}
{\sc B.~Aksoylu and H.~Klie}, {\em
  \href{http://www.cct.lsu.edu/~burak/pubs/aksoyluKlie2008.pdf} {A family of
  physics-based preconditioners for solving elliptic equations on highly
  heterogeneous media}}, Appl. Num. Math., 59 (2009), pp.~1159--1186.
\newblock doi:10.1016/j.apnum.2008.06.002.

\bibitem{AkYe2009}
{\sc B.~Aksoylu and Z.~Yeter}, {\em Robust multigrid preconditioners for
  cell-centered finite volume discretization of the high-contrast diffusion
  equation}.
\newblock Submitted to Comput. Vis. Sci., 2009.

\bibitem{BaKn1990}
{\sc N.~S. Bakhvalov and A.~V. Knyazev}, {\em A new iterative algorithm for
  solving problems of the fictitious flow method for elliptic equations},
  Soviet Math. Dokl., 41 (1990), pp.~481--485.

\bibitem{Braess.D;Peisker.P1987}
{\sc D.~Braess and P.~Peisker}, {\em A conjugate gradient method and a
  multigrid algorithm for {M}orley's finite element approximation of the
  biharmonic equation}, Numerische Mathematik, 50 (1987), pp.~567--586.

\bibitem{ciarlet2002_book}
{\sc P.~G. Ciarlet}, {\em The Finite Element Method for Elliptic Problems},
  Classics in applied mathematics, SIAM, Philadelphia, PA, 2002.

\bibitem{HCTorigPaper}
{\sc R.~W. Clough and J.~L. Tocher}, {\em Finite element stiffness matrices for
  analysis of plates in bending}, in Proceedings of the conference on matrix
  methods in structural mechanics, 1965.

\bibitem{dang2006}
{\sc Q.~A. Dang}, {\em Iterative method for solving the {N}eumann boundary
  value problem for biharmonic type equation}, J. Comput. Appl. Math, 196
  (2006), pp.~643--643.

\bibitem{GrHa:99}
{\sc I.~G. Graham and M.~J. Hagger}, {\em {Unstructured additive
  {Schwarz}-conjugate gradient method for elliptic problems with highly
  discontinuous coefficients}}, SIAM J. Sci. Comp., 20 (1999), pp.~2041--2066.

\bibitem{grossi2001}
{\sc R.~O. Grossi}, {\em On the existence of weak solutions in the study of
  anisotropic plates}, J. Sound Vibration, 242 (2001), pp.~542--552.

\bibitem{Hanisch.M1993}
{\sc M.~R. Hanisch}, {\em Multigrid preconditioning for the biharmonic
  {D}irichlet problem}, SIAM J. Numer. Anal., 30 (1993),
  pp.~184--214.

\bibitem{kato1982_book}
{\sc T.~Kato}, {\em A short introduction to perturbation theory for linear
  operators}, Springer, Berlin, Germany, 1982.

\bibitem{KnWi:2003}
{\sc A.~Knyazev and O.~Widlund}, {\em {Lavrentiev regularization + Ritz
  approximation = uniform finite element error estimates for differential
  equations with rough coefficients}}, Math. Comp., 72 (2003), pp.~17--40.

\bibitem{Maes.J;Bultheel.A2006}
{\sc J.~Maes and A.~Bultheel}, {\em A hierarchical basis preconditioner for the
  biharmonic equation on the sphere}, IMA J. Numer. Anal., 26 (2006),
  pp.~563--583.

\bibitem{manolisRangelovShaw2003}
{\sc G.~D. Manolis, T.~V. Rangelov, and R.~P. Shaw}, {\em The non-homogeneous
  biharmonic plate equation: fundamental solutions}, Internat. J. Solids
  Structures, 40 (2003), pp.~5753--5767.

\bibitem{marcinkowski2007}
{\sc L.~Marcinkowski}, {\em An additive {S}chwarz method for mortar finite
  element discretizations of the 4th order elliptic problem in {2D}}, Electron.
  Trans. Numer. Anal., 26 (2007), pp.~34--54.

\bibitem{mayo1984}
{\sc A.~Mayo}, {\em The fast solution of {Poisson's} and the biharmonic
  equations on irregular regions}, SIAM J.\ Numer.\ Anal., 21 (1984),
  pp.~285--299.

\bibitem{mayoGreenbaum1992}
{\sc A.~Mayo and A.~Greenbaum}, {\em Fast parallel iterative solution of
  {Poisson's} and the biharmonic equations on irregular regions}, SIAM J.\
  Sci.\ Statist.\ Comput., 13 (1992), pp.~101--118.

\bibitem{Mihajlovic.M;Silvester.D2004a}
{\sc M.~Mihajlovic and D.~Silvester}, {\em {A black-box multigrid
  preconditioner for the biharmonic equation}}, BIT, 44 (2004), pp.~151--163.

\bibitem{Mihajlovic.M;Silvester.D2004}
{\sc M.~Mihajlovi{\'c} and D.~Silvester}, {\em {Efficient parallel solvers for
  the biharmonic equation}}, Parallel Computing, 30 (2004), pp.~35--55.

\bibitem{millerHorgan1995}
{\sc K.~L. Miller and C.~O. Horgan}, {\em End effects for plane deformations of
  an elastic anisotropic semi-infinite strip}, J. Elasticity, 38 (1995),
  pp.~261--316.

\bibitem{morleyOrigPaper}
{\sc L.~S.~D. Morley}, {\em The triangular equilibrium problem in the solution
  for plate bending problems}, Aero. Quart., 19 (1968), pp.~149--169.

\bibitem{Oswald.P1992}
{\sc P.~Oswald}, {\em Hierarchical conforming finite element methods for the
  biharmonic equation}, SIAM J. Numerical Analysis, 29 (1992), pp.~1610--1625.

\bibitem{Oswald.P1995}
\leavevmode\vrule height 2pt depth -1.6pt width 23pt, {\em Multilevel
  preconditioners for discretizations of the biharmonic equation by rectangular
  finite elements}, Numer. Lin. Alg. Appl., 2 (1995), pp.~487--505.

\bibitem{pozrikidis2005_book}
{\sc C.~Pozrikidis}, {\em Introduction to Finite and Spectral Element Methods
  using {MATLAB}}, Chapman \& Hall/CRC, Boca Raton, FL, 2005.

\bibitem{wang2005_masterThesis}
{\sc T.~S. Wang}, {\em A {Hermite} cubic immersed finite element space for beam
  design problems}.
\newblock Master thesis, Department of Mathematics, Virginia Polytechnic
  Institute and State University, 2005.

\bibitem{watkins2002_book}
{\sc D.~S. Watkins}, {\em Fundamentals of Matrix Computations},
  Wiley-Interscience; second edition, New York, 2002.

\bibitem{Zhang.X1994}
{\sc X.~Zhang}, {\em Multilevel {S}chwarz methods for the biharmonic Dirichlet
  problem}, SIAM J. Sci. Comput., 15 (1994), pp.~621--644.

\end{thebibliography}
